\documentclass[]{article}

\usepackage{amsmath,amsfonts, amsthm, amssymb, version}
\usepackage[breaklinks,colorlinks=true,citecolor=blue,linkcolor=blue]{hyperref}
\usepackage{todonotes}
\usepackage[round,numbers,square]{natbib}
\usepackage{listings}
\usepackage{booktabs}

\allowdisplaybreaks[1]

\newtheorem{theorem}{Theorem}

\newtheorem{proposition}[theorem]{Proposition}

\theoremstyle{remark}

\theoremstyle{definition}
\newtheorem{definition}[theorem]{Definition}
\newtheorem{example}[theorem]{Example}

\newcommand{\N}{\mathbb{N}}
\newcommand{\R}{\mathbb{R}}

\newcommand{\D}{\mathcal{D}}
\newcommand{\A}{\mathcal{A}}

\newcommand{\f}{\mathfrak{f}}

\numberwithin{theorem}{section}

\numberwithin{equation}{section}

\begin{document}

\title{A Hitchhiker's Guide to Automatic Differentiation}
\date{}
\author{Philipp H. W. Hoffmann} 

\maketitle
%\pagestyle{myheadings}
%\markright{What is Forward Automatic Differentiation?}
\renewcommand{\thefootnote}{\fnsymbol{footnote}}
\begin{center} \emph{This a preprint of an article which has appeared in:\\ Numerical Algorithms, Vol. 72 No. 3 (2016), 775-811.\footnote{The final publication is available at 
www.springerlink.com\\
Journal URL: \url{http://link.springer.com/journal/11075}}} \end{center}

\renewcommand{\thefootnote}{\arabic{footnote}}
 \setcounter{footnote}{0}
\begin{abstract}
This article provides an overview of some of the mathematical principles of Automatic Differentiation (AD). In particular, we summarise different descriptions of the Forward Mode of AD, like the matrix-vector product based approach, the idea of lifting functions to the algebra of dual numbers, the method of Taylor series expansion on dual numbers and the application of the push-forward operator, and explain why they all reduce to the same actual chain of computations. We further give a short mathematical description of some methods of higher-order Forward AD and, at the end of this paper, briefly describe the Reverse Mode of Automatic Differentiation.
\end{abstract}



\section{Introduction}\label{intro}
Automatic Differentiation (short AD), also called Algorithmic or Computational Differentiation, is a method to evaluate derivatives of functions which differs significantly from the classical ways of computer-based differentiation through either approximative, numerical methods, or through symbolic differentiation, using computer algebra systems. While approximative methods (which are usually based on finite differences) are inherently prone to truncation and rounding errors and suffer from numerical instability, symbolic differentiation may (in certain cases) lead to significant long computation times. Automatic Differentiation suffers from none of these problems and is, in particular, well-suited for the differentiation of functions implemented as computer code. Furthermore, while Automatic Differentiation is also numerical differentiation, in the sense that it computes numerical values, it computes derivatives up to machine precision. That is, the only inaccuracies which occur are those which appear due to rounding errors in floating-point arithmetic or due to imprecise evaluations of elementary functions. For these reasons, AD has received significant interest from computer scientists and applied mathematicians, in the last decades.

The very first article on this procedure is probably due to Wengert \cite{Wengert1964ASA} and appeared already in 1964. Two further major publications regarding AD were published by Rall in the 1980s \cite{Rall1983Dag}, \cite{Rall1986} and, since then, there has been a growing community of researcher interested in this topic. 

So what is Automatic Differentiation? The answer to this question may be sought in one of the many publications on this topic, which usually provide a short introduction to the general theory. Furthermore, there are also excellent and comprehensive publications which describe the area as a whole (see, for example, Griewank \cite{Griewank2003AMV} and Griewank and Walther \cite{Griewank2008}). However, an unfamiliar reader may find it nevertheless difficult to grasp the essence of Automatic Differentiation. The problem lies in the diversity with which the (actual simple) ideas can be described. While in \cite{Griewank2003AMV} and \cite[Section 2]{PearlmutterSiskind2008a} the step-wise evaluation of a matrix-vector product is described as the basic procedure behind AD, in \cite{kalman-2002a} Automatic Differentiation is defined via a certain multiplication on pairs (namely the multiplication which defines the algebra of \emph{dual numbers}). Similarly, in \cite{Siskind2008NFM} the lifting of a function to said dual numbers is presented as the core principle of AD, where in \cite[Section 2]{pearlmutter-siskind-popl-2007}, the evaluation of the Taylor series expansion of a function on dual numbers appears to be the main idea. Finally, Manzyuk \cite{Manzyuk-mfps2012} bases his description on the push-forward operator known from differential geometry and, again, gives a connection to functions on dual numbers. While the latter descriptions at least appear to be similar (although not identical), certainly the matrix-vector product based approach seems to differ from the remaining methods quite a lot. Of course, all the publications mentioned contain plenty of cross-references and each different description of AD has its specific purpose. However, for somebody unfamiliar with the theory, it may still be difficult to see why the described techniques are essentially all equivalent. This article hopes to clarify the situation.

We will in the following give short overviews of the distinct descriptions of AD\footnote{To be more precise, of the Forward Mode of AD.} mentioned above and show, why they all are just different expressions of the same principle. It is clear that the purpose of this article is mainly educational and there is little intrinsically new in our elaborations. Indeed, in particular with regards to \cite{Griewank2003AMV}, we only give a extremely shorted and simplified version of the work in the original publication. Furthermore, there are actually at least two distinct versions, or modes, of AD. The so-called \emph{Forward Mode} and the \emph{Reverse Mode} (along with variants such as \emph{Checkpoint Reverse Mode} \cite{Griewank1992ALG}). The different descriptions mentioned above all refer to the Forward Mode only. We are, therefore, mainly concerned with Forward AD. We will discuss the standard Reverse Mode only in the preliminaries and in a section at the end of this paper.

In addition, we will mainly restrict ourselves to AD in its simplest form. Namely, Automatic Differentiation to compute (directional) first 
order derivatives, of a differentiable, multivariate function $f: X \to \R^m$, on an open set $X \subset \R^n$. We only briefly discuss the 
computation of higher-order partial derivatives in Section \ref{Higher-Order Partial}, referring mainly to the works of Berz \cite{Berz} and 
Karczmarczuk \cite{KarczmarczukI}. There is also a rich literature on the computation of 
whole Hessians (see, for instance, \cite{Dixon} or \cite{GowerMello}), however, we will not be concerned with this extension of AD in this 
article. The same holds for Nested Automatic Differentiation, which involves a kind of recursive calling of AD (see, for 
example, \cite{Siskind2008NFM}). Again, we will not be concerned with this topic in this paper. 

As mention above, Automatic Differentiation is often (and predominantly) used to differentiate computer programs, that is, implementations of mathematical functions as code. In the case of first order Forward AD, the mathematical principle used is usually the lifting of functions to dual numbers (see Figure \ref{Figure Haskell code Dual} for an implementation example with test case). More information on this topic can, for example, be found in \cite{Bischof}, \cite{Griewank2008} or (in particular considering higher-order differentiation) in \cite{KarczmarczukI}.

The notation we are using is basically standard. As mentioned above, the function we want to differentiate will be denoted by $f$ and will be defined on an open set $X \subset \R^n$ (denoted by $U$ in Section \ref{Higher-Order Partial} to avoid confusion).\footnote{In principle, one can also consider functions of complex variables. Since the rules of real and complex differential calculus are the same, this does not lead to any changes in the theory.} In particular, in this paper $n$ always denotes the number of variables of $f$, while $m$ denotes 
the dimension of its co-domain. In Sections \ref{section Griewank} and \ref{section Reverse}, the notation $x_i$ is 
reserved for variables of the function $f$, while other variables are denoted by $v_i$. The symbol $c$ always denotes a fixed value (a constant). For real vectors, we use boldface letters like $\mathbf{x}$ or $\mathbf{c}$ (where the latter will be a constant vector). Furthermore, 
$\mathbf{\overset{\rightharpoonup}{x}}$ and $\mathbf{\overset{\leftharpoonup}{y}}$ will be (usually fixed) directional vectors or $1$-row matrices, respectively. Entries of $\mathbf{\overset{\rightharpoonup}{x}}$ or $\mathbf{\overset{\leftharpoonup}{y}}$ will be denoted by $x'_i$ or $y'_i$, respectively\footnote{The notation $x'_i$ for entries of 
$\mathbf{\overset{\rightharpoonup}{x}}$ is somewhat historical and based on the idea that, very often, $x'_i$ may be considered as a derivative of either the identity function, or a constant function. For us, however, each $x'_i \in \R$ is simply a chosen real number. The same holds for the notation $y'_i$.}. Finally, we denote all multiplications (of numbers, as well as matrix-vector multiplication) mostly by a simple dot. The symbol $*$ will be used sometimes when we want to emphasize that multiplication of numbers is a differentiable function on $\R^2$. 

\section{Preliminaries}

\subsection{The basic ideas of Automatic Differentiation}\label{subsection basic idea}

Before we start with the theory, let us demonstrate the ideas of AD in a very easy case: Let $f, \varphi_1, \varphi_2, \varphi_3: \R \to \R$ be differentiable functions with $f = \varphi_3 \circ \varphi_2 \circ \varphi_1$. Let further $c, x', y' \in \R$ be real numbers. Assume we want to compute $f'(c) \cdot x'$ or $y' \cdot f'(c)$, respectively. (Of course, the distinction between multiplication from the left and from the right is motivated by the more general case of multivariate functions.)

By the chain rule, 
\begin{align*}
&f'(c) \cdot x' = \varphi'_3\left(\varphi_2\left(\varphi_1(c)\right)\right) \cdot \varphi'_2\left(\varphi_1(c)\right) \cdot \varphi'_1(c) \cdot x' 
\end{align*}
As one easily sees, the evaluation of $f'(c) \cdot x'$ can be achieved by computing successively the following pairs of real numbers:
\begin{figure}
\centering
\includegraphics[]{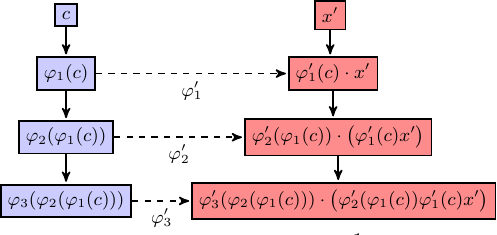}
\caption{Computational graph for the computation of $f'(c) \cdot x'$.}\label{first graph}
\end{figure}

\[\begin{array}{c}(c,\ x') \\ (\varphi_1(c),\ \varphi'_1(c) \cdot x') \\ (\varphi_2\left(\varphi_1(c)\right), \ \varphi'_2\left(\varphi_1(c)\right) \cdot \varphi'_1(c) x') \\ 
(\varphi_3\left(\varphi_2\left(\varphi_1(c)\right)\right), \ \varphi'_3\left(\varphi_2\left(\varphi_1(c)\right)\right) \cdot \varphi'_2\left(\varphi_1(c)\right) \varphi'_1(c)  x')\end{array}\] and taking the second entry of the final pair. As we see, the first element of each pair appears as an argument of the functions $\varphi_i, \varphi'_i$ in the following pair, while the second element appears as a factor (from the right) to the second element in the following pair. 

Regarding the computation of $y' \cdot f'(c)$, we have obviously
\begin{align*}
&y' \cdot f'(c) = y' \cdot \varphi'_3\left(\varphi_2\left(\varphi_1(c)\right)\right) \cdot \varphi'_2\left(\varphi_1(c)\right) \cdot \varphi'_1(c).
\end{align*}
The computation of this derivative can now be achieved by the computing the following two lists of real numbers:
\[\begin{array}{c} c \\ \varphi_1(c) \\ \varphi_2(\varphi_1(c)) \\ \varphi_3(\varphi_2(\varphi_1(c))) \end{array} \ \ \ \ \begin{array}{c} y' \\ y' \cdot \varphi'_3\left(\varphi_2\left(\varphi_1(c)\right)\right) \\ y' \varphi'_3\left(\varphi_2\left(\varphi_1(c)\right)\right) \cdot \varphi'_2\left(\varphi_1(c)\right) \\ y'  \varphi'_3\left(\varphi_2\left(\varphi_1(c)\right)\right) \varphi'_2\left(\varphi_1(c)\right) \cdot \varphi'_1(c)\end{array} \]
and taking the last entry of the second list. Here, each entry (apart from $y'$) in the second list consists of values of 
$\varphi'_i$ evaluated at an element of the first list (note that the order is reversed) and the previous entry as a factor (from the left).
\begin{figure}
\centering
\includegraphics[]{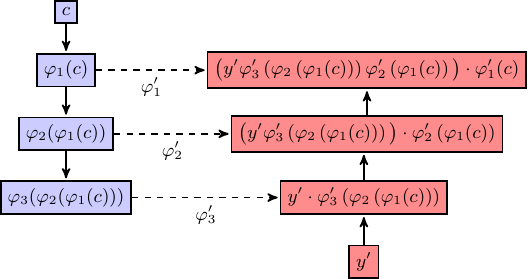}
\caption{Computational graph for the computation of $y'\cdot f'(c)$.}\label{second graph}
\end{figure}

In both examples, the computation of $\varphi_3(\varphi_2(\varphi_1(c)))$ is actually unnecessary to obtain the sought derivative. This value is, however, computed in all models we will consider in this article.

If now the functions $\varphi_1, \varphi_2, \varphi_3$ and their derivatives $\varphi'_1, \varphi'_2, \varphi'_3$ are implemented in the system, then the evaluation of values $\varphi_i(v_i),\varphi'_i(v_i)$ for some $v_i$ means simply calling these functions/derivatives with suitable inputs. The computation of $f'(c) \cdot x'$ or $y' \cdot f'(c)$ then becomes nothing else than obtaining values $\varphi_i(v_i),\varphi'_i(v_i)$, performing a multiplication and passing the results on. That is, neither is some derivative evaluated symbolically, nor is some differential or difference quotient computed. In that sense, the derivative of $f$ is computed `automatically'.

\subsection{The setting in general}

As mentioned above, (First Order) Automatic Differentiation, in its simplest form, is concerned with the computation of derivatives of a differentiable function
$f: X \to \R^m$, on an open set $X \subset \R^n$. The assumption made is that each $f_j: X \to \R$ in 
\[f(x_1,...,x_n)=\left( \begin{array}{c} f_1(x_1,...,x_n) \\ \vdots \\ f_m(x_1,...,x_n) \end{array} \right), \ \ 
\textrm{for all} \ (x_1,...,x_n) \in X,\]
%=\left( \begin{array}{c} f^{[1]}(x_1,...,x_n) \\ \vdots \\ f^{[m]}(x_1,...,x_n) \end{array} \right)\] 
consists (to be defined more precisely later) of several sufficiently smooth so-called \emph{elementary} (or \emph{elemental}) \emph{functions} $\varphi_i: U_i \to \R $, defined on open sets $U_i \subset \R^{n_i}$, with $i \in I$ for some index set $I$. The set of elementary functions $\{\varphi_i \ | \ i \in I\}$ has to be given and can, in principle, consist of arbitrary functions as long as these are sufficiently often differentiable. However, certain functions are essential for computational means, including addition and multiplication\footnote{Here, we consider indeed addition and multiplication as differentiable functions
\\ $+: \R^2 \to \R$ and $*: \R^2 \to \R$.}, constant, trigonometric, exponential functions etc. Figure \ref{table essential functions} shows a table of such a (minimal) list. A more comprehensive list can be found, for example, in \cite[Table 2.3]{Griewank2008}.
\begin{figure}
\centering
\includegraphics[]{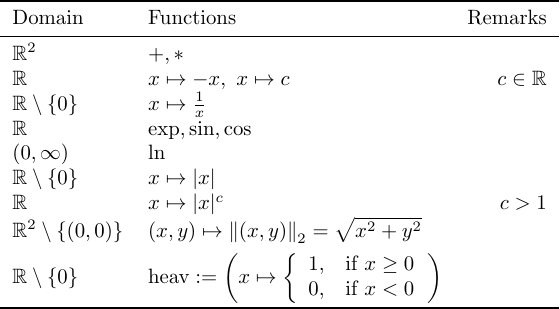}
\caption{Table of essential elementary functions according to Griewank and Walther \cite{Griewank2008}. The domains are chosen such that the functions are differentiable.}\label{table essential functions}
\end{figure}
All elementary functions will be implemented in the system together with their gradients. 

Automatic Differentiation now does not compute the actual mapping\\ 
$\mathbf{x} \mapsto J_f(\mathbf{x})$, which maps a vector $\mathbf{x}\in X$ to the Jacobian $J_f(\mathbf{x})$ of $f$ at $\mathbf{x}$. Instead, directional derivatives of $f$ or left-hand products of row-vectors with its Jacobian 
at a fixed vector $\mathbf{c}\in X$ are determined. That is, given $\mathbf{c}\in X$ and $ \mathbf{\overset{\rightharpoonup}{x}} \in \R^n $ or $\mathbf{\overset{\leftharpoonup}{y}} \in \R^{1 \times m}$, we determine either 
\[J_f(\mathbf{c}) \cdot \mathbf{\overset{\rightharpoonup}{x}} \ \ \ \ \ \textrm{or} \ \ \ \ \ \mathbf{\overset{\leftharpoonup}{y}}  \cdot J_f(\mathbf{c}).\] 
(This is not a subtle difference, since, while $\mathbf{x} \mapsto J_f(\mathbf{x})$ is a matrix-valued function, $J_f(\mathbf{c}) \cdot \mathbf{\overset{\rightharpoonup}{x}}$ and $ \mathbf{\overset{\leftharpoonup}{y}}  \cdot J_f(\mathbf{c}) $ are vectors or one-row matrices, respectively, in euclidean space.)

The computation of directional derivatives of $J_f(\mathbf{c})$ is referred to as the Forward Mode of AD, or Forward AD, while the computation 
of $\mathbf{\overset{\leftharpoonup}{y}} \cdot J_f(\mathbf{c})$ is referred to as the Reverse Mode of AD, or Reverse AD. We may give the following, informal descriptions:

Let $f: X \to \R^m$ consist of (not necessarily distinct!) elementary functions $\varphi_1,...,\varphi_\mu$. Then
\begin{itemize}
\item
\emph{Forward Automatic Differentiation is the computation of $J_f(\mathbf{c}) \cdot \mathbf{\overset{\rightharpoonup}{x}}$ for fixed $\mathbf{c}\in X$ and $ \mathbf{\overset{\rightharpoonup}{x}} \in \R^n $ through the successive computation of pairs of real numbers 
\[\left(\varphi_1(\mathbf{c_{1}} ),\ \nabla  \varphi_1(\mathbf{c_{1}}  ) \cdot   \mathbf{\overset{\rightharpoonup}{x_{1}}}\right),...,
\left(\varphi_\mu(\mathbf{c_{\mu}} ),\ \nabla  \varphi_\mu(\mathbf{c_{\mu}}  ) \cdot   \mathbf{\overset{\rightharpoonup}{x_{\mu}}}\right)
\in \R^2\]
for suitable vectors $\mathbf{c_{i}} \in U_i, \mathbf{\overset{\rightharpoonup}{x_{i}}} \in \R^{n_i}$, $i=1,...,\mu$.}
\item
\emph{Reverse Automatic Differentiation is the computation of $\mathbf{\overset{\leftharpoonup}{y}} \cdot J_f(\mathbf{c})$ for fixed  $\mathbf{c}\in X$ and $\mathbf{\overset{\leftharpoonup}{y}} \in \R^{1 \times m}$ through the computation of the two lists of real numbers 
\begin{align*}
&\varphi_1(\mathbf{c_{1}}),...,\varphi_\mu(\mathbf{c_{\mu}})\in \R \\[1ex] 
\textrm{and} \ \ \ &v_{1} \cdot 
\frac{\partial \varphi_1}{\partial  v_k}(\mathbf{c_{1}}) + v_{1,k}\ ,...,\ v_{\mu} \cdot 
\frac{\partial \varphi_\mu}{\partial  v_k}(\mathbf{c_{\mu}}) + v_{\mu,k} \ \in \R, \ \ k = 1,...,n_{i},
\end{align*}
for suitable vectors $\mathbf{c_{i}} \in U_i$ and suitable numbers $v_{i},v_{i,k} \in \R$, $i=1,...,\mu$.}
\end{itemize}
Of course, the vectors and numbers $\mathbf{c_{i}}, \mathbf{\overset{\rightharpoonup}{x_{i}}},v_{i}, v_{i,k}$ are determined in a certain way; as is the order of in which the computations are performed.

As mentioned before, the function $f$ has to be constructed using elementary functions. Loosely speaking, we may say that $f$ has to be a composition of elements of $\{\varphi_i \ | \ i \in I\}$. However, this is not quite correct from a strictly mathematically point of view. Since all elementary functions are real-valued, it is clear that a composition $\varphi_\mu \circ \cdots \circ \varphi_1$ can not be defined, as soon as one of the $\varphi_2,...,\varphi_\mu$ is multivariate.\footnote{For instance, it is impossible to write $x \mapsto \exp(x) + \sin(x)$ as a composition of $\exp, \sin$ and $+$.} Admittedly, this is a rather technical and not really important issue, but, for completeness, we give the following inductive definition:

\begin{definition}\label{definition automatically differentiable}
\begin{itemize}
\item[(i)]
We call a function $h: X \to \R$ on open $X \subset \R^{n}$ \emph{automatically differentiable}, if
\begin{itemize}
\item $h \in \{ \varphi_i \ | \ i \in I\}$ or
\item there exist functions $h_k: X_k \to \R$ on open sets $X_k \subset \R^{n_k}$, $k=1,...,\ell$, such that for all $\mathbf{x}=(x_1,...,x_n) \in X$, there exist $n_0 \geq 0$ many $x_{0,1},...,x_{0,n_0} \in \{x_1,...,x_{n}\}$ and, for $k=1,...,\ell-1$, $l=1,...,n_k$, there exist $n_k$ many $x_{k,l} \in \{x_1,...,x_{n}\} \cap X_k$,  with
\begin{align*}
&h(\mathbf{x}) \\
&= h_{\ell}(x_{0,1},...,x_{0,n_0},h_1(x_{1,1},...,x_{1,n_1}),...,h_{\ell-1}(x_{\ell-1,1},...,x_{\ell-1,n_{\ell-1}})),
\end{align*}
and $h_{\ell} \in \{ \varphi_i \ | \ i \in I\}$ and, for $k=1,...,\ell-1$, each $h_k$ is automatically differentiable.
\end{itemize}
\item[(ii)] We call a function $f: X \to \R^m$ with $f(\mathbf{x})=\left( \begin{array}{c} f_1(\mathbf{x}) \\ \vdots \\ f_m(\mathbf{x}) \end{array}\right)$ for all $\mathbf{x} \in X$ \emph{automatically differentiable}, if each $f_j: X \to \R$ is automatically differentiable.
\end{itemize}
\end{definition}

\begin{example}\label{example automatically differentiable}
The function
$h: \R^2 \to \R$ given by
\begin{align*}
h(x_1,x_2) &= \sin(x_2) + 5 * \cos(x_1 * x_1).
\end{align*}
is automatically differentiable.\begin{comment}

If we set $h_3 = + $, $h_1 = \sin$ and $h_2: \R \to \R$ with $h_2(x_1)= 5\cos(x_1^2)$, we have
\[h(x_1,x_2) = h_3(h_1(x_2),h_2(x_1)),\]
where clearly $h_3$ and $h_1$ are elementary. So consider $h_2$:

Set $h_{2,3} = * $, $h_{2,1}: \R \to \R$ with $h_{2,1}(x_1) = 5$ and $h_{2,2}: \R \to \R$ with $h_{2,2}(x_1)= \cos(x_1^2)$. Then
\[h_2(x_1) = h_{2,3}(h_{2,1}(x_1), h_{2,2}(x_1)),\]
where $h_{2,3}$ and $h_{2,1}$ are elementary. So consider $h_{2,2}$:

Setting $h_{2,2,2} = \cos$ and $h_{2,2,1} = *$ gives
\[h_{2,2}(x_1) = h_{2,2,2}(h_{2,2,1}(x_1,x_1)),\]
where clearly, both, $h_{2,2,2}$ and $h_{2,2,1}$ are elementary functions. 

Thus, by definition, $h_{2,2}$ is automatically differentiable, which makes $h_2$ automatically differentiable and, hence, $h$ is automatically differentiable.\end{comment}
\end{example}

From now on, without necessarily stating it explicitly, we will always assume that our function $f: X \to \R^m$ is automatically differentiable in the sense of Definition \ref{definition automatically differentiable}. 

One may, rightfully, ask why we use an inductive description of automatically differentiable functions, instead of just describing them as compositions of suitable multi-variable, \underline{multi-dimensional} mappings. However, from a computational point of view, one should note that an automatically differentiable $f: X \to \R^m$ will usually be given in the form of Definition \ref{definition automatically differentiable}, such that expressing $f$ as a composition may require additional work. 

Nevertheless, expressing $f$ as a composition is indeed the basic step in an elementary description of Automatic Differentiation, which we describe in the next Section. We will describe other (equivalent) approaches which work directly with functions of the form of Definition \ref{definition automatically differentiable} in later sections.

\section{Forward AD---An elementary approach}\label{section Griewank}

In this approach, the function $f$ is described as a composition of multi-variate and multi-dimensional mappings. Differentiating this composition to obtain $J_f(\mathbf{c}) \cdot \mathbf{\overset{\rightharpoonup}{x}}$, for given 
$\mathbf{c}\in X \subset \R^n$ and $\mathbf{\overset{\rightharpoonup}{x}} \in \R^n$, leads, by the chain rule, to a product of matrices. This method has, for example, been described in \cite[Section 2]{PearlmutterSiskind2008a} and, comprehensively, in the works of Griewank \cite{Griewank2003AMV} and Griewank and Walther \cite{Griewank2008}. We follow mainly the notation of \cite{Griewank2003AMV}.

The simple idea is to express $f$ as a composition of the form
\[ f=P_Y \circ \Phi_{\mu} \circ \cdots \circ \Phi_1 \circ P_X.\] 
Here, $P_X: X \to H$ is the (linear) natural embedding of the domain $X \subset \R^n$ into the so-called \emph{state 
space} $H := \R^{n+ \mu}$, where $\mu$ is the total number of (not necessarily distinct) elementary functions 
$\varphi_i$ of which $f$ consists. Each \\
$\Phi_i: H \to H$, referred to as an 
\emph{elementary transition}, corresponds to exactly one such elementary function. The mapping 
$P_Y : H \to \R^m$ is some suitable linear projection of $H$ down into $\R^m$.

Determining now $J_f(\mathbf{c}) \cdot \mathbf{\overset{\rightharpoonup}{x}}$ for fixed $\mathbf{c}\in X$ and fixed
$\mathbf{\overset{\rightharpoonup}{x}} \in \R^n$ becomes, by the chain rule, the evaluation of the matrix-vector product
\begin{align}\label{matrix-vector product Griewank}
J_f(\mathbf{c}) \cdot \mathbf{\overset{\rightharpoonup}{x}} &= 
P_Y \cdot \Phi'_{\mu,\mathbf{c}}\ \cdots \ \Phi'_{1,\mathbf{c}} \cdot P_X \cdot \mathbf{\overset{\rightharpoonup}{x}},\end{align}
where $\Phi'_{i,\mathbf{c}}$ denotes the Jacobian of $\Phi_i$ at $\left(\Phi_{i-1}\circ \cdots \circ \Phi_1 \circ P_X \right)(\mathbf{c})$. 

%Since each $\Phi_i$ corresponds to one elementary function $\varphi_i$, the computation of $\Phi'_{i,\mathbf{c}}$ only involves the computations (which means calling) of partial derivatives of $\varphi_i$ at some $\mathbf{c_i} \in U_i \subset \R^{n_i}$. Similarly, computing the vectors 
%$\mathbf{c_i}$ requires the calling of the elementary functions $\varphi_{i-1},...,\varphi_1$. This way, $J_f(\mathbf{c}) \cdot \mathbf{\overset{\rightharpoonup}{x}}$ is computed through calling elementary functions and derivatives and some matrix-vector multiplication. That is, `automatically' while, at no time in this process, any symbolic differentiation takes place.

%evaluating directional derivatives of the $\varphi_{i}$ by   
%invoking the $\nabla \varphi_i$ with these inputs and
%evaluating the matrix-vector product in \eqref{matrix-vector product Griewank}. That is, at not time in this process, any actual differentiation takes place. Instead, a series of arithmetic operations on real numbers are performed. The result 
%is, however, the directional derivative $J_f(\mathbf{c}) \cdot \mathbf{\overset{\rightharpoonup}{x}}$.

The process is now performed in a particular ordered fashion, which we describe in the following.

The evaluation of $f$ at some point $\mathbf{x}=(x_1,...,x_n)$ can be described by a so-called 
\emph{evaluation trace}
$\mathbf{v}^{[0]}=\mathbf{v}^{[0]}(\mathbf{x}),...,\mathbf{\mathbf{v}}^{[\mu]}=\mathbf{v}^{[\mu]}(\mathbf{x})$, where each $\mathbf{\mathbf{v}}^{[i]} \in H$ is a so-called \emph{state vector}, representing the state of the evaluation after $i$ steps. More precisely, we set
\[\mathbf{\mathbf{v}}^{[0]}:= P_X(x_1,...,x_n) = (x_1,...,x_n,0,...,0) \ \ \textrm{and} \ \ \mathbf{\mathbf{v}}^{[i]} = \Phi_i(\mathbf{\mathbf{v}}^{[i-1]}), \ \ i = 1,...,\mu.\]
The elementary transitions $\Phi_i$ are now given by imposing some suitable ordering on the $\mu$ elementary functions $\varphi_k$ of which $f$ consists, such that $\varphi_i$ is the $i$-th elementary function with respect to this order, and by setting 
\[ \Phi_i\left(\begin{array}{c} v_1 \\ \vdots \\ v_{n+\mu}\end{array}\right) 
= \left(\begin{array}{c} v_1 \\ \vdots \\ v_{n+i-1} \\ \varphi_i(v_{i_1},...,v_{i_{n_i}}) 
											\\ v_{n+i+1} 
											\\ \vdots 
											\\ v_{n+\mu} \end{array} \right),\ \ \textrm{for all} \  
											\left(\begin{array}{c} v_1 \\ \vdots \\ v_{n+\mu}\end{array}\right) \in H,\]
and $v_{i_1},...,v_{i_{n_i}} \in \{v_1,...,v_{n+i-1}\} \cap U_i$ (where $U_i \subset \R^{n_i}$ is the open domain of $\varphi_i$). Note that this is not a definition in the strict sense, since we neither specify the ordering of the $\varphi_i$, nor the arguments $v_{i_1},...,v_{i_{n_i}}$ of each $\varphi_i$. These will depend on the actual functions $f$ and $\varphi_i$. (Compare the example below.)

Therefore, we have 					
\begin{align*}\mathbf{v}^{[i]}(\mathbf{x}) = \Phi_i(\mathbf{\mathbf{v}}^{[i-1]}(\mathbf{x})) = \left(\begin{array}{c} \mathbf{\mathbf{v}}^{[i-1]}_1(\mathbf{x}) = x_1
											\\ \vdots 
											\\ \mathbf{\mathbf{v}}^{[i-1]}_n(\mathbf{x}) = x_n
											\\ \vdots
											\\ \mathbf{\mathbf{v}}^{[i-1]}_{n+i-1} (\mathbf{x})
											\\ \varphi_i(v_{i_1}(\mathbf{x}),...,v_{i_{n_i}}(\mathbf{x})) 
											\\ 0 
											\\ \vdots 
											\\ 0 \end{array} \right)
 = \left(\begin{array}{c} \mathbf{\mathbf{v}}^{[i-1]}_1 = x_1
											\\ \vdots 
											\\ \mathbf{\mathbf{v}}^{[i-1]}_n = x_n
											\\ \vdots
											\\ \mathbf{\mathbf{v}}^{[i-1]}_{n+i-1} 
											\\ \varphi_i(v_{i_1},...,v_{i_{n_i}}) 
											\\ 0 
											\\ \vdots 
											\\ 0 \end{array} \right), \end{align*} 
for $v_{i_1}= v_{i_1}(\mathbf{x}),...,v_{i_{n_i}} = v_{i_{n_i}}(\mathbf{x}) \in 
\{\mathbf{\mathbf{v}}^{[i-1]}_1,...,\mathbf{\mathbf{v}}^{[i-1]}_{n+i-1}\} \cap U_i$.

It is clear that, for the above to make sense, the ordering imposed on the elementary functions $\varphi_k$ must have the property 
that all arguments in $\varphi_i(v_{i_1},...,v_{i_{n_i}})$ have already been evaluated, before $\varphi_i$ is applied.

The definition of the projection $P_Y:H \to \R^m$ depends on the ordering imposed on the $\varphi_k$. If this ordering is such that we have 
%for the top-level elementary functions $\varphi^{[1]}_{\ell_1}=\varphi_{n+\mu-m} , \ ..., \  \varphi^{[m]}_{\ell_m}=\varphi_{n+\mu}$, then we will have 
\[f_1(x_1,...,x_n) = \mathbf{\mathbf{v}}^{[\mu]}_{n+\mu-m} \ ,...,\ f_m(x_1,...,x_n) = \mathbf{\mathbf{v}}^{[\mu]}_{n+\mu},\]
we can obviously choose $P_Y(v_1,...,v_{n+\mu}) = (v_{n+\mu-m},...,v_{n+\mu})$.

\begin{example}
The following is a trivial modification of an example taken from \cite[page 332]{Griewank2003AMV}.

Consider the function $f: \R^2 \to \R^2$ given by
\[ f(x_1,x_2) = \left( \begin{array}{c}\exp(x_1) * \sin(x_1 + x_2) \\ x_2 \end{array}\right).\]

Choose $H = \R^{7}$ and $f = P_Y \circ \Phi_5 \circ \Phi_4 \circ \Phi_3 \circ \Phi_2 \circ \Phi_1 \circ P_X $ with 
\[
P_X: \R^2 \to \R^7, \ \ \textrm{with} \ \ P_X(x_1,x_2) = (x_1,x_2, 0, 0, 0, 0, 0),\]
$\Phi_i: \R^7 \to \R^7$, $i=1,...,5$, with
\begin{align*}
&\Phi_1\left(v_1,v_2,v_3,v_4,v_5,v_6,v_7 \right) 
= (v_1,v_2,\exp(v_1),v_4,v_5,v_6,v_7), \\ 
&\Phi_2\left(v_1,v_2,v_3,v_4,v_5,v_6,v_7 \right)
= (v_1, v_2, v_3, v_1 + v_2,v_5,v_6,v_7 ),\\
&\Phi_3\left(v_1,v_2,v_3,v_4,v_5,v_6,v_7 \right)
= (v_1, v_2, v_3, v_4, \sin(v_4),v_6,v_7 ),\\
&\Phi_4\left(v_1,v_2,v_3,v_4,v_5,v_6,v_7 \right)= (v_1, v_2, v_3, v_4, v_5, v_3 * v_5,v_7),\\
&\Phi_5\left(v_1,v_2,v_3,v_4,v_5,v_6,v_7 \right)= \left(v_1,v_2,v_3,v_4,v_5,v_6,v_2 \right) 
\end{align*}
and 
\[P_Y: \R^7 \to \R^2, \ \ \textrm{with} \ \ P_Y(v_1,v_2,v_3,v_4,v_5,v_6,v_7) = (v_6,v_7).\]
\end{example}

Analogously to the evaluation of $f(\mathbf{x})$, the evaluation of the matrix-vector product 
\eqref{matrix-vector product Griewank} for some $\mathbf{c} \in \R^n$ and some
$\mathbf{\overset{\rightharpoonup}{x}} = (x_1',...,x_n') \in \R^n$ can be expressed as an evaluation trace $\mathbf{v'}^{[0]}=\mathbf{v'}^{[0]}(\mathbf{c},\mathbf{\overset{\rightharpoonup}{x}}),...,\mathbf{v'}^{[\mu]}=\mathbf{v'}^{[\mu]}(\mathbf{c},\mathbf{\overset{\rightharpoonup}{x}})$, where
\[\mathbf{v'}^{[0]} := P_X \cdot \mathbf{\overset{\rightharpoonup}{x}}= (x_1',...,x_m',0,...,0)  \ \ \ \textrm{and} \ \ \ \mathbf{\mathbf{v'}}^{[i]}: = \Phi'_{i,\mathbf{c}} \cdot \mathbf{v'}^{[i-1]}, \ \ i = 1,...,\mu.\]
By the nature of the elementary transformations $\Phi_i$, each Jacobian\\ $\Phi'_{i,\mathbf{c}} := J_{\Phi_i}(\mathbf{v}^{[i-1]}(\mathbf{c}))$ will be of the form
\begin{align}\label{Phi'}
&\Phi'_{i, \mathbf{c}} = \left(\begin{array}{cccccc} 
1 & \cdots & 0 & 0 & \cdots & 0 \\													
\vdots & \ddots & \vdots & \vdots & \cdots & \vdots \\
0 & \cdots & 1 & 0 & \cdots & 0 \\
\frac{\partial \varphi_i}{\partial  v_1}(\cdots ) & \cdots & \cdots & \cdots & \cdots & \frac{\partial \varphi_i}{\partial  v_{n+\mu}}
(\cdots ) \\
0 & \cdots & 0 & 1 & \cdots & 0 \\
\vdots & \vdots & \vdots & \vdots & \ddots & \vdots \\
0 & \cdots & 0 & 0 & \cdots & 1 
\end{array}\right) \leftarrow \ (n+i)\textrm{-th row},
\end{align}
where $\frac{\partial \varphi_i}{\partial  v_{k}}
(\cdots ) = \frac{\partial \varphi_i}{\partial  v_k}(v_{i_1}(\mathbf{c}),...,v_{i_{n_i}}(\mathbf{c}) )$ is interpreted as $0$ if $\varphi_i$ does not depend on $v_k$.

Thus, each $\mathbf{\mathbf{v'}}^{[i]}$ will be of the form
\[\mathbf{\mathbf{v'}}^{[i]} = \left(\begin{array}{c} \mathbf{\mathbf{v'}}^{[i-1]}_1 = x'_1
											\\ \vdots 
											\\ \mathbf{\mathbf{v'}}^{[i-1]}_n = x'_n
											\\ \vdots
											\\ \mathbf{\mathbf{v'}}^{[i-1]}_{n+i-1} 
											\\ \nabla \varphi_i(v_{i_1},...,v_{i_{n_i}} ) 
											\cdot \left( \begin{array}{c} v'_{i_1} \\ \vdots \\ v'_{i_{n_i}} \end{array} 
											\right)
											\\ 0 
											\\ \vdots 
											\\ 0 \end{array} \right),\]
for $v'_{i_1}=v'_{i_1}(\mathbf{c},\mathbf{\overset{\rightharpoonup}{x}}),...,v'_{i_{n_i}} 
= v'_{i_{n_i}}(\mathbf{c},\mathbf{\overset{\rightharpoonup}{x}}) \in \{\mathbf{\mathbf{v'}}^{[i-1]}_1,...,\mathbf{v'}^{[i-1]}_{n+i-1}\}$, where
the \\
$v'_{i_1},..., v'_{i_{n_i}}$ correspond exactly to the $v_{i_1},..., v_{i_{n_i}}$. That is, if
$v_{i_j} = \mathbf{v}^{[i-1]}_l(\mathbf{c})$, then $v'_{i_j} = \mathbf{v'}^{[i-1]}_l(\mathbf{c},\mathbf{\overset{\rightharpoonup}{x}})$.

The directional derivative of $f$ at $\mathbf{c}$ in direction of $\mathbf{\overset{\rightharpoonup}{x}}$ is then simply
\[J_{f}(\mathbf{c}) \cdot \mathbf{\overset{\rightharpoonup}{x}} = P_Y \cdot \mathbf{v'}^{[\mu]}.\]

\begin{example}\label{Example griewank}
Let $c_1,c_2,x'_1,x'_2 \in \R$. The computation of $J_f((c_1,c_2)) \cdot \left(\begin{array}{c} x'_1 \\ x'_2 \end{array}\right) $ with $f: \R^2 \to \R^2$ given by
\[f(x_1,x_2) = \left( \begin{array}{c}\exp(x_1) * \sin(x_1 + x_2) \\ x_2 \end{array}\right) \]
has five \emph{evaluation trace pairs} $[\mathbf{v}^{[0]},\mathbf{v'}^{[0]}],...,[\mathbf{v}^{[5]},\mathbf{v'}^{[5]}]$, where
\begin{align*}
 \mathbf{v}^{[0]} = (c_1, c_2, 0, 0, 0, 0, 0)\ \ \textrm{and} \ \ \mathbf{v'}^{[0]} =(x'_1, x'_2, 0, 0, 0, 0, 0)
\end{align*}
%\[ \mathbf{v}^{[0]} = \left(\begin{array}{c} c_1 \\ c_2 \\ 0 \\ 0 \\ 0 \\ 0 \\ 0\end{array}\right), \ \ \    \mathbf{v'}^{[0]} = \left(\begin{array}{c} x'_1 \\ x'_2 \\ 0 \\ 0 \\ 0 \\ 0 \\ 0\end{array}\right) \]
and 
\begin{align*}
\mathbf{\mathbf{v}}^{[5]}= 
\left(\begin{array}{c} c_1 \\ c_2 \\ \exp(c_1) \\ c_1 + c_2  \\ \sin(c_1+c_2) \\ \exp(c_1) * \sin(c_1+c_2) \\ c_2 \end{array} \right),
\end{align*}
\begin{align*}
\mathbf{\mathbf{v'}}^{[5]}= 
\left(\begin{array}{c} x'_1 \\ x'_2 \\ \exp(c_1)x'_1 \\ x'_1 + x'_2  \\ \cos(c_1+c_2)(x'_1+x'_2) \\ 
\sin(c_1+c_2)\exp(c_1)x'_1 + \exp(c_1)\cos(c_1+c_2)(x'_1+x'_2) \\ x'_2\end{array} \right).
\end{align*}
Then
$J_f((c_1,c_2)) \cdot \left( \begin{array}{c} x'_1 \\ x'_2 \end{array} \right) = P_Y \cdot \mathbf{\mathbf{v'}}^{[5]}$, which is
\begin{align*} \left( \begin{array}{c} \big(\sin(c_1+c_2)\exp(c_1) + \exp(c_1)\cos(c_1+c_2)\big)x'_1 + \exp(c_1)\cos(c_1+c_2)x'_2 \\ x'_2 \end{array}\right).
\end{align*}
\begin{figure}
\centering
\includegraphics[scale = 0.99]{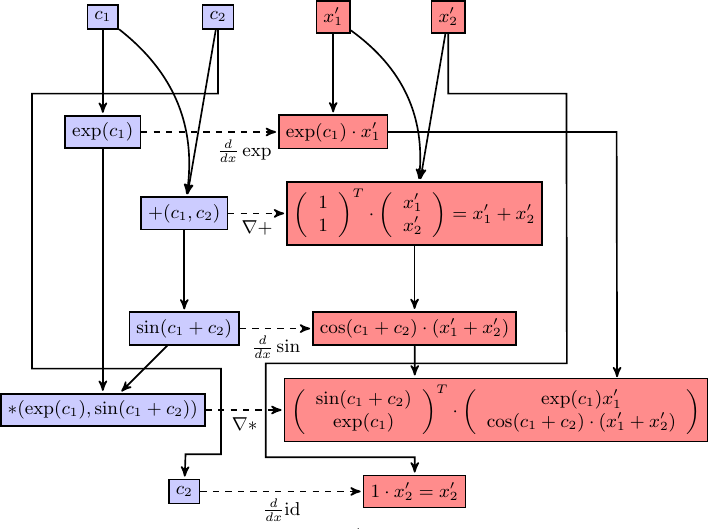}
\caption{Computational graph for Example \ref{Example griewank} with elements of $\mathbf{v}^{[5]}$ in blue and of $\mathbf{v'}^{[5]}$ in red.}\label{graph for griewank}
\end{figure}
\end{example}

Note that in the evaluation process, given the $\Phi_i$, each pair $[\mathbf{v}^{[i]},\mathbf{v'}^{[i]}]$ depends only on the previous pair $[\mathbf{v}^{[i-1]},\mathbf{v'}^{[i-1]}]$ and the given vectors 
$\mathbf{c}, \mathbf{\overset{\rightharpoonup}{x}}$. (Since $\mathbf{v}^{[i]} = \Phi_i(\mathbf{v}^{[i-1]})$ and 
$\mathbf{v'}^{[i]}=J_{\Phi_i}(\mathbf{v}^{[i-1]}(\mathbf{c})) \cdot \mathbf{v'}^{[i-1]}$.) Therefore, in an implementation, one can actually overwrite $[\mathbf{v}^{[i-1]},\mathbf{v'}^{[i-1]}]$ by 
$[\mathbf{v}^{[i]},\mathbf{v'}^{[i]}]$ in each step.

Note further that the $(n+i)$-th entry in each pair $[\mathbf{v}^{[i]},\mathbf{v'}^{[i]}]$ is of the form 
\begin{align}\label{Griewank pairs}\left(\varphi_i(v_{i_1},...,v_{i_{n_i}} ), \nabla  \varphi_i(v_{i_1},...,v_{i_{n_i}} ) \cdot \left( \begin{array}{c} v'_{i_1} \\ \vdots \\ v'_{i_{n_i}} \end{array} \right)\right) \in \R^2,
\end{align} i.e. consisting of a value of $\varphi_i$ and a directional derivative of this elementary function. Since the previous $n+i-1$ entries are identical to the first\\ $n+i-1$ entries of $[\mathbf{v}^{[i-1]},\mathbf{v'}^{[i-1]}]$, the computation of $[\mathbf{v}^{[i]},\mathbf{v'}^{[i]}]$ is effectively the computation of \eqref{Griewank pairs}.

We summarize the discussion of this Section:

\begin{theorem}
By the above, given $\mathbf{c} \in X \subset \R^n$ and $\mathbf{\overset{\rightharpoonup}{x}} \in \R^n$, the evaluation of $J_f(\mathbf{c}) \cdot \mathbf{\overset{\rightharpoonup}{x}}$ of an automatically differentiable function $f: X \to \R^m$ can be achieved by computing the evaluation trace 
pairs $[\mathbf{v}^{[i]},\mathbf{v'}^{[i]}]$. This process is equivalent to the computation of the pairs \eqref{Griewank pairs}.
\end{theorem}

The following section is concerned with a method which uses this last fact directly from the start. The approach about to be described also provides a better understanding on how an Automatic Differentiation system could actually be implemented. A question which may not be quite clear from the discussion so far.

\section{Forward AD---An approach using Dual Numbers}\label{section dual numbers}

Many descriptions and implementation of Forward AD actually use a slightly different approach than the elementary one that we have just described. Instead of expressing the function whose derivative one wants to compute as a composition, the main idea in this `alternative' approach\footnote{Indeed, we will see at the end of this Section, that Forward AD using dual numbers is completely equivalent to the method of expressing $f$ as $P_Y \circ \Phi_{\mu} \circ \cdots \circ \Phi_1 \circ P_X$.} is to lift this function (and all elementary functions) to (a subset of) the algebra of \emph{dual numbers} $\mathcal{D}$.  This method has, for example, been described in \cite{kalman-2002a}, \cite{pearlmutter-siskind-popl-2007} and \cite{Rall1986}.

Dual numbers, introduced by Clifford \cite{Clifford1873}, are defined as $\mathcal{D} := (\R^2, +, \cdot)$, where addition is defined component-wise, as usual, and multiplication is defined as 
\[(x_1,y_1) \cdot (x_2,y_2) := (x_1x_2, x_1y_2+y_1x_2), \ \ \ \forall \ (x_1,y_1),(x_2,y_2) \in \R^2.\]
It is easy to verify that $\D$ with these operations is an associative and commutative algebra over $\R$ with multiplicative unit $(1,0)$ and that the element $\varepsilon := (0,1)$ is nilpotent of order two. \footnote{$\varepsilon$ is sometimes referred to as an infinitesimal in the literature. The correct interpretation of this is probably that one can replace dual numbers by elements from non-standard analysis in the context of AD. However, this approach is actually unnecessary and, given the complexity of non-standard analysis, we will not consider it here.} 

Analogously to a complex number, we write a dual number $z = (x,y)$ as $z = x + y \varepsilon$, where we identify each $x \in \R$ with $(x,0)$. We will further use the notation $(x,x')$ instead of $(x,y)$, i.e. we write $z = x + x' \varepsilon$. The $x'$ in this representation will be referred to as the \emph{dual part} of $z$.

We now define an extension of a differentiable, real-valued function\\
$h: X \to \R$, defined on open $X \subset \R^n$, to a 
function $\widehat{h}: \D^{n} \supset X \times \R^{n} \to \D$ defined on a subset of the dual numbers, by setting
\begin{align}\label{definition extension dual numbers}
\widehat{h}(x_1+x_1' \varepsilon, ...,x_{n}+x'_{n}\varepsilon ) := h(x_1,...,x_{n}) 
+\left( \nabla h(x_1,...,x_{n}) \cdot \left( \begin{array}{c} x'_1 \\ \vdots \\ x'_{n} \end{array} \right) \right) \cdot \varepsilon.
\end{align}
This definition easily extends to differentiable functions $f: X \to \R^m$, where\\ 
$\widehat{f}: \D^n \supset X \times \R^n \to \D^m$ is defined via
\begin{align}\label{extension dual numbers mulit-dimensional}\widehat{f}(x_1+x_1' \varepsilon, ...,x_{n}+x'_{n}\varepsilon)
&:=\left( \begin{array}{c} \widehat{f}_1(x_1+x_1' \varepsilon, ...,x_{n}+x'_{n}\varepsilon) \\ \vdots \\ \widehat{f}_m(x_1+x_1' \varepsilon, ...,x_{n}+x'_{n}\varepsilon) \end{array} \right)\\[1ex]\notag
&\ = f(x_1,...,x_n) + \left( J_f(x_1,...,x_n) \cdot \left( \begin{array}{c} x'_1 \\ \vdots \\ x'_{n} \end{array} \right) \right)\varepsilon.
\end{align}

The following statement shows that definition \eqref{definition extension dual numbers} makes sense. I.e., that it is compatible with the natural extension of functions which are defined via usual arithmetic, i.e. polynomials, and analytic functions. That is:

\begin{proposition}\label{prop. dual numbers def. makes sense}
Definition \eqref{definition extension dual numbers} is compatible with the ``natural'' extensions of 
\begin{itemize}\item[(i)] real-valued constant functions
\item[(ii)] projections of the form $(x_1,..,x_n) \mapsto x_k$,
\item[(iii)] the arithmetic operations $+$, $*$ $: \R^2 \to \R$ and\\ 
$/: \{(x_1,x_2) \in \R^2 \ | \ x_2 \neq 0\} \to \R$,  with
\begin{align*}
+(x_1,x_2):= x_1+x_2, \ \ \ *(x_1,x_2):=x_1 \cdot x_2, \ \ \ /(x_1,x_2):= \frac{x_1}{x_2}\ .
\end{align*}
\item[(iv)] (multivariate) polynomials and rational functions
\item[(v)] (multivariate) real analytic functions
\end{itemize}
to subsets of $\D$, $\D^2$ or $\D^n$, respectively.
\end{proposition}
\begin{proof}
(i) and (ii) follow easily from the definition.

(iii): We have
\begin{align*}
\widehat{+}\big(x_1+x'_1\varepsilon ,x_2+x'_2 \varepsilon \big) 
=& +(x_1,x_2) + \left( \nabla +(x_1,x_2) \cdot \left( \begin{array}{c} x_1' \\ x_{2}' \end{array} \right) \right)\varepsilon \\
&= (x_1+x_2) + \left( \left( \begin{array}{c} 1 \\ 1 \end{array} \right)^T \cdot 
\left( \begin{array}{c} x_1' \\ x_{2}' \end{array} \right) \right)\varepsilon \\
&= (x_1+x_2) + (x'_1+x'_2)\varepsilon\\
&= (x_1+x'_1 \varepsilon) + (x_2+x'_2 \varepsilon) 
\end{align*} 
and, since $ \varepsilon^2=0 $,
\begin{align*}
\widehat{*}(x_1+x'_1\varepsilon ,x_2+x'_2 \varepsilon ) 
&= *(x_1, x_2) + \left( \nabla *(x_1,x_2) \cdot \left( \begin{array}{c} x_1' \\ x_{2}' \end{array} \right) \right)\varepsilon \\ 
&= (x_1 \cdot x_2) + \left( \left( \begin{array}{c} x_2 \\ x_1 \end{array} \right)^T \cdot 
\left( \begin{array}{c} x_1' \\ x_{2}' \end{array} \right) \right)\varepsilon \\
&= (x_1 \cdot x_2) + (x_2 x'_1 + x_1 x'_2)\varepsilon  \\
&= (x_1 + x'_1 \varepsilon) \cdot (x_2 + x'_2 \varepsilon).
\end{align*}
Finally, considering division, it is easy to see that the multiplicative inverse of a dual numbers $x + x' \varepsilon$ is defined if, and only if, $x \neq 0$ and given by $\frac{1}{x} + \left( - \frac{x'}{x^2} \right) \varepsilon $.

Then, for $x_2 \neq 0$, since $ \varepsilon^2=0 $,
\begin{align*}
\widehat{/}(x_1+x'_1\varepsilon, x_2+x'_2\varepsilon)
&= /(x_1, x_2) + \left( \nabla / (x_1, x_2) \cdot \left( \begin{array}{c} x_1' \\ x_{2}' \end{array} \right) \right)\varepsilon \\ 
&= \frac{x_1}{x_2} + \left( \left( \begin{array}{c} \frac{1}{x_2} \\[1ex] -\frac{x_1}{x_2^2} \end{array} \right)^T \cdot 
\left( \begin{array}{c} x_1' \\ x_{2}' \end{array} \right) \right)\varepsilon \\ 
&= \frac{x_1}{x_2} + \left( \frac{x'_1}{x_2} - \frac{x_1\cdot x'_2}{x_2^2} \right) \varepsilon\\
&= (x_1 + x'_1 \varepsilon) \cdot \left( \frac{1}{x_2} + \left( - \frac{x'_2}{x_2^2} \right) \varepsilon \right) \\
&= \frac{x_1 + x'_1 \varepsilon}{x_2 + x'_2 \varepsilon} \ . 
\end{align*}

(iv): This will follow from Proposition \ref{proposition dual numbers quasi-compositions} in connection with (iii).

(v): We will recall the definition of multi-variate Taylor series in Section \ref{section Dual and Taylor}. For the moment, let $T_k(h;\mathbf{c})$ denote the $k$-th degree (multi-variate) Taylor polynomial of $h: X \to \R$ about $\mathbf{c} \in X$. Since $h$ is real analytic, we have  
\[T_k(h;\mathbf{c})(x_1,...,x_{n}) \to  h(x_1,...,x_n) \ \ \ (k \to \infty)\]
for all $(x_1,...,x_{n}) \in V$, where $V$ is an open neighbourhood of $\mathbf{c}$. It is well-known, that then 
\[\frac{\partial}{\partial x_j} T_k(h;\mathbf{c})(x_1,...,x_{n}) \to \frac{\partial}{\partial x_j} h(x_1,...,x_n) \ \ \ (k \to \infty) \]
on $V$ (see for example \cite[Chapter II.1]{John1975}). Since addition and multiplication are continuous, then also 
\[\left( \nabla T_k(h;\mathbf{c})(x_1,...,x_{n}) \cdot \mathbf{\overset{\rightharpoonup}{x}} \right) \to \left( \nabla h(x_1,...,x_n) \cdot \mathbf{\overset{\rightharpoonup}{x}} \right) \ \ \ (k \to \infty) \]
on $V$, for any fixed $\mathbf{\overset{\rightharpoonup}{x}} = (x'_1,...,x'_n) \in \R^n$. Consequently, 
\[\widehat{T_k}(h;\mathbf{c})(x_1+x'_1 \varepsilon,...,x_{n}+x'_n \varepsilon) \to  \widehat{h}(x_1+x'_1 \varepsilon,...,x_n+x'_n \varepsilon) \ \ \ (k \to \infty),\]
for all $(x_1+x'_1 \varepsilon,...,x_{n}+x'_n \varepsilon) \in V \times \R^n$. 
\end{proof}

To use definition \eqref{definition extension dual numbers} for Automatic Differentiation, we need to show that it behaves well for automatically differentiable functions as defined in Definition \ref{definition automatically differentiable}. Basically, we need to show that \eqref{definition extension dual numbers} is compatible with the chain rule.

\begin{proposition}\label{proposition dual numbers quasi-compositions}
Let $h: X \to \R$ defined on open $X \subset \R^{n}$ be automatically differentiable, $h \notin \{\varphi_i \ | \ i \in I\}$.
Then 
\begin{align}\label{lemma lift to dual numbers}\notag
&\widehat{h}(\mathbf{x}+\mathbf{\overset{\rightharpoonup}{x}}\varepsilon) \\ \notag
&= \widehat{h}_{\ell}(x_{0,1}+ \varepsilon x'_{0,1} ,...,x_{0,n_0}+ \varepsilon x'_{0,n},\widehat{h}_1(x_{1,1}+ \varepsilon x'_{1,1},...,x_{1,n_1}+ \varepsilon x'_{1,n_1}),\\ 
&\ \ \ \ \ ...,\widehat{h}_{\ell-1}
(x_{\ell-1,1}+ \varepsilon x'_{\ell-1,1},...,x_{\ell-1,n_{\ell-1}}+ \varepsilon x'_{\ell-1,n_{\ell-1}})),
\end{align}
for all $\mathbf{x}+\varepsilon \mathbf{\overset{\rightharpoonup}{x}} := (x_1 + \varepsilon x'_1,...,x_n + \varepsilon x'_n) \in X \times \R^n$, with \\
$x_{k,i} + x'_{k,i} \in \{x_1 + \varepsilon x'_1,...,x_n + \varepsilon x'_n\} \cap X_k$.
\end{proposition}
\begin{proof}
We prove this statement by direct computation.\footnote{A more elegant proof can be given by writing $h$ as a composition and using the fact that the push-forward operator (see Section \ref{section diffgeom}) is a functor.}
In the following, denote 
\begin{align*}&\mathbf{x_k} := (x_{k,1},...,x_{k,n_{k}}) \in X_k \subset \R^{n_k} \ \ 
\textrm{and} \ \ &\mathbf{\overset{\rightharpoonup}{x}_k} := (x'_{k,1} ,...x'_{k,n_{k}}) \in X_k \times \R^{n_k}.
\end{align*}
Then the right hand-side of equation \eqref{lemma lift to dual numbers} is equal to 
\begin{align}\label{lemma right hand side}
\notag
&\widehat{h_{\ell}}\left(\mathbf{x_0}, h_1(\mathbf{x_1})+\left( \nabla h_1(\mathbf{x_1}) \cdot \mathbf{\overset{\rightharpoonup}{x}_k} \right)
\varepsilon, \right. \\ &\left. 
\ \ \ \ \ ...,
h_{\ell-1}(\mathbf{x_{\ell-1}})+\left( \nabla h_{\ell-1}(\mathbf{x_{\ell-1}}) \cdot \mathbf{\overset{\rightharpoonup}{x}_{\ell-1}}\right)\varepsilon\right)\\ \notag
&= h_{\ell}\left(\mathbf{x_0},h_1(\mathbf{x_1}),...,h_{\ell-1}(\mathbf{x_{\ell-1}})\right)
+ \left( \nabla h_{\ell}\left(...\right) \cdot 
\left(\begin{array}{c} x'_{0,1} \\ \vdots \\  x'_{0,n} \\ \nabla h_1(\mathbf{x_1}) \cdot \mathbf{\overset{\rightharpoonup}{x}_1} \\ \vdots \\ 
\nabla h_{\ell-1}(\mathbf{x_{\ell-1}}) \cdot \mathbf{\overset{\rightharpoonup}{x}_{\ell-1}} \end{array}\right) \right) \varepsilon,
\end{align} 
where
\begin{align*}
\nabla h_{\ell}\left(...\right) &:= \nabla h_{\ell} \left(\mathbf{x_0},h_1(\mathbf{x_1}),...,h_{\ell-1}(\mathbf{x_{\ell-1}}) \right)\\
&=\frac{dh_{\ell}}{d(\mathbf{x_0},h_1(\mathbf{x_1}),...,h_{\ell-1}(\mathbf{x_{\ell-1}}))}(\mathbf{x_0},h_1(\mathbf{x_1}),...,h_{\ell-1}
(\mathbf{x_{\ell-1}})).
\end{align*}

The left hand side of \eqref{lemma lift to dual numbers} is obviously equal to
\begin{align}\label{lemma left hand side}
&h(\mathbf{x}) +\left( \nabla h(\mathbf{x}) \cdot \mathbf{\overset{\rightharpoonup}{x}} \right) \cdot \varepsilon.
\end{align}
By assumption, $h(\mathbf{x}) = h_{\ell}(\mathbf{x_0},h_1(\mathbf{x_1}),...,h_{\ell-1}(\mathbf{x_{\ell-1}}))$. 
Further, by the chain rule,
\begin{align*}
\nabla h(\mathbf{x})  \cdot \mathbf{\overset{\rightharpoonup}{x}} &= \frac{d h_{\ell}}{d\mathbf{x}}(\mathbf{x_0},h_1(\mathbf{x_1}),...,h_{\ell-1}(\mathbf{x_{\ell-1}}))  \cdot \mathbf{\overset{\rightharpoonup}{x}} \\[1ex]
&= \nabla h_{\ell} \left(\mathbf{x_0},h_1(\mathbf{x_1}),...,h_{\ell-1}(\mathbf{x_{\ell-1}}) \right)\\
&\ \ \ \ \cdot \frac{d\left( \mathbf{x} \mapsto (\mathbf{x_0},h_1(\mathbf{x_1}),...,h_{\ell-1}(\mathbf{x_{\ell-1}})\right)}{d\mathbf{x}}
\left(\mathbf{x} \right) \cdot \mathbf{\overset{\rightharpoonup}{x}}.
\end{align*}
Now, $\frac{d\left( \mathbf{x} \mapsto (\mathbf{x_0},h_1(\mathbf{x_1}),...,h_{\ell-1}(\mathbf{x_{\ell-1}})\right)}{d\mathbf{x}}
\left(\mathbf{x} \right) $ equals
\begin{align*} \left( \begin{array}{ccc} \frac{\partial (\mathbf{x} \mapsto x_{0,1})}{\partial x_1}(\mathbf{x}) & 
\cdots & \frac{\partial (\mathbf{x} \mapsto x_{0,1})}{\partial x_{n}}(\mathbf{x}) \\
 \vdots & \vdots & \vdots \\
 \frac{\partial (\mathbf{x} \mapsto x_{0,n_0})}{\partial x_1}(\mathbf{x}) & \cdots & \frac{\partial (\mathbf{x} \mapsto x_{0,n_0})}{\partial x_{n}}(\mathbf{x}) \\[1ex]
 \frac{\partial (\mathbf{x} \mapsto h_1(\mathbf{x_1}))}{\partial x_1}(\mathbf{x}) & \cdots & \frac{\partial (\mathbf{x} \mapsto h_1(\mathbf{x_1}))}{\partial x_{n}}(\mathbf{x}) \\[1ex]
 \vdots & \vdots & \vdots \\
 \frac{\partial (\mathbf{x} \mapsto h_{\ell-1}(\mathbf{x_{\ell-1}}))}{\partial x_1}(\mathbf{x}) & \cdots & \frac{\partial (\mathbf{x} \mapsto h_{\ell-1}(\mathbf{x_{\ell-1}}))}{\partial x_{n}}(\mathbf{x}) \end{array} \right) 
\end{align*}
Hence,
\begin{align*}
\frac{d\left( \mathbf{x} \mapsto (\mathbf{x_0},h_1(\mathbf{x_1}),...,h_{\ell-1}(\mathbf{x_{\ell-1}})\right)}{d\mathbf{x}}
\left(\mathbf{x} \right) \cdot \mathbf{\overset{\rightharpoonup}{x}}
= \left(\begin{array}{c} x'_{0,1} \\ \vdots \\  x'_{0,n} \\ \nabla h_1(\mathbf{\overset{\rightharpoonup}{x}_1}) \cdot \mathbf{\overset{\rightharpoonup}{x}_1} \\ \vdots \\ \nabla h_{\ell-1}(\mathbf{x_{\ell-1}}) \cdot \mathbf{\overset{\rightharpoonup}{x}_{\ell-1}} \end{array}\right)
\end{align*}
Thus, \eqref{lemma right hand side} equals \eqref{lemma left hand side} and we are done. 
\end{proof}

We can now automatically compute directional derivatives $J_f(\mathbf{c}) \cdot \mathbf{\overset{\rightharpoonup}{x}}$ of an automatically differentiable function $f: X \to \R^m$, on open $X \subset \R^n$, at fixed $\mathbf{c} = (c_1,...,c_n) \in \R^n$ in direction of fixed 
$\mathbf{\overset{\rightharpoonup}{x}} = (x_1',..., x_n') \in \R^n$ by computing the directional derivatives
$\nabla f_j(\mathbf{c}) \cdot \mathbf{\overset{\rightharpoonup}{x}}$ in the following way:

\begin{theorem}\label{thm. result process dual numbers}
Assume that definition \ref{definition extension dual numbers} is implemented for all elementary functions in the set $\{\varphi_i \ |\ i \in I\}$. Then the directional derivative
$\nabla f_j(\mathbf{c}) \cdot \mathbf{\overset{\rightharpoonup}{x}}$ of an automatically differentiable function
$f_j: X \to \R$ can be computed `automatically' through extending $f_j$ to $X \times \R^n \subset \D^n$, and evaluating the dual part of $\widehat{f}_j(c_1+x'_1\varepsilon,...,c_n+x'_n\varepsilon)$.
\end{theorem}
\begin{proof}
Each $f_j$ is automatically differentiable. By assumption, the case $f_j \in \{\varphi_i \ | \ i \in I\}$ is clear: We simply obtain the pair
$\widehat{f}_j(c_1+x'_1\varepsilon,...,c_n+x'_n\varepsilon)$ by calling $f_j$ and $\nabla f_j$ and computing the 
gradient-vector product. The sought directional derivative is the dual part (second entry) of that pair. 

So assume that there exists real-valued $h_1,...,h_{\ell}$ on open sets $X_k \subset \R^{n_k}$, such that for all 
$\mathbf{x} \in X$, 
\[f_j(\mathbf{x}) = h_{\ell}(x_{0,1},...,x_{0,n_0},h_1(x_{1,1},...,x_{1,n_1}),...,h_{\ell-1}(x_{\ell-1,1},...,x_{\ell-1,n_{\ell-1}})),\]
for suitable $x_{k,j}$, with $h_{\ell} \in \{ \varphi_i \ | \ i \in I\}$ and $h_1,...,h_{\ell-1}$ automatically differentiable. 

We proceed by induction on the depth of $f_j$.

Base case: Assume that each $h_1,...,h_\ell \in \{\varphi_i \ | \ i \in I\}$.  Extending $f_j$ to $X \times \R^n \subset \D^n$ leads to the extension of $h_1,...,h_\ell$ to sets $X_k \times \R^{n_k} \subset \D^{n_k}$. Since definition \ref{definition extension dual numbers} is implemented for all functions in $\{\varphi_i \ |\ i \in I\}$, 
\[\widehat{h}_k(c_{k,1}+x'_{k,1}\varepsilon,...,c_{k,n_k}+x'_{k,n_k}\varepsilon)\]
is defined for all $h_k$ and computed by calling $h_k$ and all $\nabla h_k$ with suitable inputs and computing the 
gradient-vector product. By Proposition \ref{proposition dual numbers quasi-compositions}, the computation of
\[\widehat{h}_{\ell}(c_{0,1}+x'_{0,1}\varepsilon,...,c_{0,n_0}+x'_{0,n_0}\varepsilon, \widehat{h}_1(\cdots),...,\widehat{h}_{\ell-1}(\cdots)),\]
which is performed last, gives $\widehat{f}_j(c_1+x'_1\varepsilon,...,c_n+x'_n\varepsilon)$, whose dual part is $\nabla f_j(\mathbf{c}) \cdot \mathbf{\overset{\rightharpoonup}{x}}$.

Induction step: Assume that $h_{j_0} \in \{h_1,...,h_{\ell-1}\}$ is not an elementary function. Again, we extend $f_j$ to $X \times \R^n \subset \D^n$, which leads to the extension of $h_1,...,h_\ell$ to sets $X_k \times \R^{n_k} \subset \D^{n_k}$. Since $h_{j_0}$ is still automatically differentiable
\[\widehat{h}_{j_0}(c_{j_0,1}+x'_{j_0,1}\varepsilon,...,c_{j_0,n_{j_0}}+x'_{j_0,n_{j_0}}\varepsilon)\] 
is computed by Induction Assumption. Then again, the computation of 
\[\widehat{h}_{\ell}(c_{0,1}+x'_{0,1}\varepsilon,...,c_{0,n_0}+x'_{0,n_0}\varepsilon, \widehat{h}_1(\cdots),...,\widehat{h}_{\ell-1}(\cdots))\]
(note that $h_{\ell}$ is elementary) gives, by Proposition \ref{proposition dual numbers quasi-compositions}, the dual number $\widehat{f}_j(c_1+x'_1\varepsilon,...,c_n+x'_n\varepsilon)$.
\end{proof}

\begin{example}\label{Example dual}
Consider the function $f: \R^2 \to \R$ given by
\begin{align*}
f(x_1,x_2) &= x_2 * \cos(x_1^2+3) = x_2 * \cos(x_1 * x_1+3).
\end{align*}
Let $c_1,c_2,x'_1,x'_2 \in \R$. We evaluate the value of $\widehat{f}$ at $c_1 + x'_1 \varepsilon$ and $c_2 + x'_2 \varepsilon$. By definition \eqref{definition extension dual numbers}, Proposition \ref{prop. dual numbers def. makes sense} and Proposition \ref{proposition dual numbers quasi-compositions},
\begin{align*}
&\widehat{f}(c_1 + x'_1 \varepsilon, c_2 + x'_2 \varepsilon) \\
&= (c_2 + x'_2 \varepsilon) * \cos((c_1 +x'_1 \varepsilon)^2+3)\\
&= (c_2 + x'_2 \varepsilon) * \cos(c_1^2+2c_1x'_1\varepsilon + 3)\\
&= (c_2 + x'_2 \varepsilon) * \cos((c_1^2+3)+2c_1x'_1\varepsilon)\\
&= (c_2 + x'_2 \varepsilon) * (\cos(c_1^2+3)-\sin(c_1^2+3)2c_1x'_1\varepsilon)\\
&= c_2\cos(c_1^2+3) + \left( - 2c_1c_2\sin(c_1^2+3)x'_1 + \cos(c_1^2+3)x'_2 \right) \varepsilon.
\end{align*}
By Theorem \ref{thm. result process dual numbers}, the dual part of this expression is $\nabla f(c_1,c_2) \cdot \left( \begin{array}{c} x'_1 \\ x'_2 \end{array} \right)$. That is,
\[\nabla f(c_1,c_2) \cdot \left( \begin{array}{c} x'_1 \\x'_2 \end{array} \right) = - 2c_1c_2\sin(c_1^2+3)x'_1 + \cos(c_1^2+3))x'_2.\] 
\end{example}
\begin{figure}
\centering
\includegraphics[]{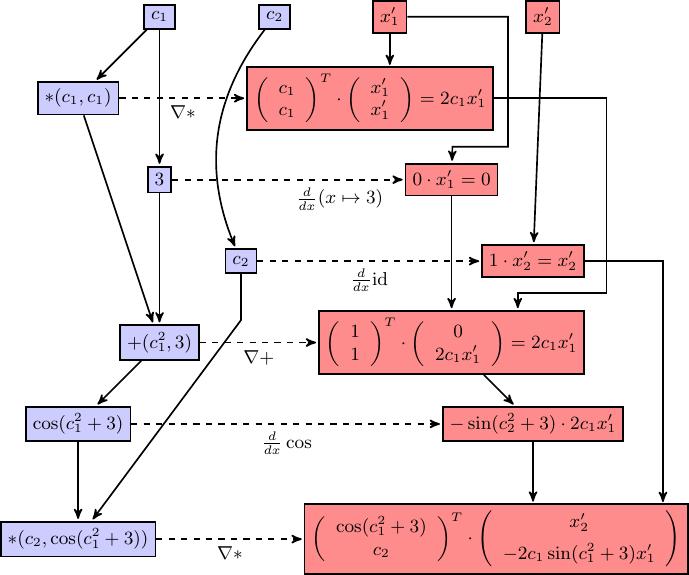}
\caption{Computational graph for Example \ref{Example dual} with primal parts in blue and dual parts in red.}\label{graph for example dual}
\end{figure}

\begin{figure}
\centering
\includegraphics[scale = 0.8]{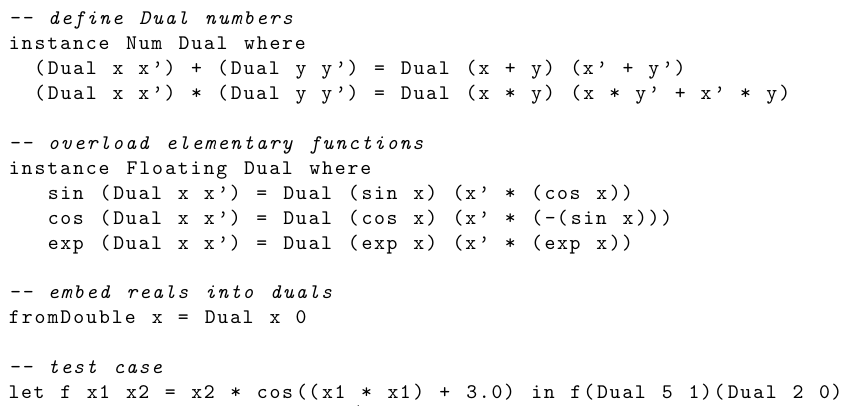}
\caption{Minimal Haskell example, showing the implementation of Forward AD for Example \ref{Example dual} and the test case $\frac{\partial f}{\partial x_1}(5,2)$. Compare also the basically identical work in \cite[Subsection 2.1]{KarczmarczukI}.}\label{Figure Haskell code Dual}
\end{figure}

Thus, a (basic) implementation of an Automatic Differentiation System can be realised by implementing \eqref{definition extension dual numbers} for all elementary functions. Usually, this is done by simply overloading elementary functions. Constant functions will usually be identified with real numbers, which themselves will be lifted to dual number with zero dual part (see Figure \ref{Figure Haskell code Dual}).

Note again that at no time during the described process any symbolic differentiation takes place. Instead, since each `top-level' function $h_{\ell}$ of an automatically differentiable function is elementary, we are computing and passing on pairs
\begin{align}\label{pairs elementary functions in dual numbers Section}
\left( \varphi_i(c_{i,1},...,c_{i,n_i}), \nabla \varphi_i(c_{i,1},...,c_{i,n_i}) \cdot \left( \begin{array}{c} x'_{i,1} \\ \vdots \\ x'_{i,n_i} \end{array} \right)\right) \in \R^2,
\end{align}
which computes `automatically' the directional derivatives $\nabla f_j(\mathbf{c}) \cdot \mathbf{\overset{\rightharpoonup}{x}}$ and, therefore, 
the directional derivative $J_f(\mathbf{c}) \cdot \mathbf{\overset{\rightharpoonup}{x}}$.

Note further that the pairs \eqref{pairs elementary functions in dual numbers Section} are \underline{exactly the same pairs} as the ones in \eqref{Griewank pairs}. This means that the processes described in this and in the previous Section reduce to \underline{exactly the same} computations.

Indeed, if we store the pairs $(c_i,x'_i)$ and \eqref{pairs elementary functions in dual numbers Section} in an array, we obtain the evaluation trace pairs $[\mathbf{v}^{[i]},\mathbf{v'}^{[i]}]$. In summary: 
 
\begin{theorem}
By the above, given $\mathbf{c} \in X \subset \R^n$ and $\mathbf{\overset{\rightharpoonup}{x}} \in \R^n$, the evaluation of $J_f(\mathbf{c}) \cdot \mathbf{\overset{\rightharpoonup}{x}}$ of an automatically differentiable function $f: X \to \R^m$ can be achieved through the lifting of each $f_j$ to a function $\widehat{f}_j: X \times \R^n \to \D$ as defined in \eqref{definition extension dual numbers} and by evaluating 
$\widehat{f}_j(c_1+x'_1\varepsilon,...,c_n+x'_n\varepsilon)$. This process is equivalent to the computation of the evaluation trace pairs $[\mathbf{v}^{[i]},\mathbf{v'}^{[i]}]$, as described in the previous section, and to the computation of the pairs \eqref{pairs elementary functions in dual numbers Section} in suitable order.
\end{theorem}

\section{Comparison with Symbolic Differentiation---Some Thoughts on Complexity}\label{section comparison FAD with symbolic}

Before we move on to further descriptions of Forward AD and to a brief description of the Reverse Mode, we take a look at some example cases to demonstrate how Automatic Differentiation solves some complexity issues which appear in systems which use symbolic differentiation. 

Assume that we want to obtain the derivative of a composition $f$ of uni-variate elementary functions $\varphi_1,...,\varphi_n: \R \to \R$, that is,
\[f= \varphi_n \circ \cdots \circ \varphi_1,\]
at a certain value $c \in \R$.

A symbolic differentiation system will first use the chain rule to determine the derivative function $f'$, which is given by
\[f'(x) = \varphi'_n(\varphi_{n-1}(\cdots (\varphi_1(x)))) \cdot \varphi'_{n-1}(\varphi_{n-2}(\cdots (\varphi_1(x))))\cdot \ \cdots\ \cdot \varphi'_2(\varphi_1(x)) \cdot \varphi'_1(x),\]
for all $x \in \R$,
and then compute $f'(c)$ by substituting $x$ by $c$. That is, each factor is computed and the results are multiplied. Hence, the system computes the following values:
\begin{align*}
&\varphi'_1(c) \\ 
&\varphi_1(c), \ \varphi'_2(\varphi_1(c))\\
&\varphi_1(c), \ \varphi_2(\varphi_1(c)), \ \varphi'_3(\varphi_2(\varphi_1(c))) \\
&\varphi_1(c), \ \varphi_2(\varphi_1(c)), \ \varphi_3(\varphi_2(\varphi_1(c))), \ \varphi'_4(\varphi_3(\varphi_2(\varphi_1(c))))\\
&\vdots \\
&\varphi_1(c), \ \varphi_2(\varphi_1(c)), \ \varphi_3(\varphi_2(\varphi_1(c))), ... , \ \varphi'_n(\varphi_{n-1}(\cdots (\varphi_1(c))))
\end{align*}
As we see many expressions will be computed multiple times (\emph{loss of sharing}). If we ignore the time the system needs to determine $f'$, as well as the time for performing multiplications, and set the cost for the computation of each value of $\varphi_i, \varphi'_i$ as $1$, then the total cost of computing 
$f'(c)$ is 
\[1 + 2 + 3 +\cdots + n = \sum_{k=1}^n k = \frac{n(n+1)}{2} \in \mathcal{O}(n^2). \]

In comparison, a Forward Automatic Differentiation system will perform a computation of pairs starting with $\big( \varphi_1(c), \varphi'_1(c) \cdot 1 \big)$, where the $(i+1)$-st pair for $i \geq 1$ looks like
\[\big(\varphi_{i+1}(v_i), \ \varphi'_{i+1}(v_i) \cdot v'_i\big),\]
for some $v_i, v'_i \in \R$. See the computational graph in Figure \ref{first graph} in Subsection \ref{subsection basic idea} for the case $n=3$ (set $x'=1$).
If we again ignore costs for multiplications, the cost for evaluating each pair is $2$. Hence, the total costs of evaluating $f'(c)$ via Automatic Differentiation is $2n$. 

In \cite{Manzyuk-mfps2012} the example of a product of the form
\[f = \varphi_n * \cdots * \varphi_1 = * (\varphi_n, \ * (\varphi_{n-1}, *(\cdots (*(\varphi_2,\varphi_1)))))\]
is given. Here, the evaluation of the derivative of $f'(c)$ at some $c \in \R$ by a symbolic differentiation system will first use the product rule to compute $f'$ given by
\begin{align*}f'(x) = &\varphi'_n(x) \cdot (\varphi_{n-1}(x) \cdots \varphi_1(x)) + \varphi'_{n-1}(x) \cdot (\varphi_{n}(x) \cdot \varphi_{n-2}(x) \cdots \varphi_1(x))\\ 
&+ \cdots + \varphi'_1(x) \cdot (\varphi_{n}(x) \cdots \varphi_2(x)),
\end{align*}
for all $x \in \R$, and then again substitute $x$ by $c$. Again, many function values will be computed multiple times. Since we have $n$ functions in each summand and $n$ summands, ignoring cost for multiplications and addition, the computation of $f'(c)$ has a total cost of $n^2$.

In contrast, a Forward Automatic Differentiation system will compute pairs starting with 
\[\big(\varphi_1(c), \ \varphi'_1(c) \cdot 1\big), \ \big(\varphi_2(c), \ \varphi'_2(c) \cdot 1\big),\ \left(\varphi_1(c) * \varphi_2(c),\ \left(\begin{array}{c} \varphi_2(c) \\ \varphi_1(c) \end{array}\right)^T \cdot \left(\begin{array}{c} \varphi_1'(c) \\ \varphi'_2(c) \end{array}\right)\right)
\]
where the remaining pairs for $i \geq 2$ look like
\begin{align*}
&\big(\varphi_{i+1}(c), \ \varphi'_{i+1}(c) \cdot 1\big), \ \ 
&\left(\varphi_{i+1}(c) * v_i, \ \left(\begin{array}{c}v_i \\  \varphi_{i+1}(c)\end{array}\right)^T \cdot 
\left(\begin{array}{c}\varphi'_{i+1}(c) \\ v'_{i} \end{array}\right)\right)
\end{align*}
for some $v_i, v'_i \in \R$. Since the $3$rd, $5$th, $7$th etc. pairs are those which are created by lifting $*$ to the dual numbers, they contain only additions and multiplications of values which have already been computed. Hence, for simplicity, we may discard these pairs with regards to the costs of the evaluations of $f'(c)$. Thus, the total cost of computing $f'(c)$ is the cost of computing the $n$ pairs $\big(\varphi_{i+1}(c), \ \varphi'_{i+1}(c) \cdot 1\big)$ which is $2n$. 
\begin{figure}
\centering
\includegraphics[scale = 0.9]{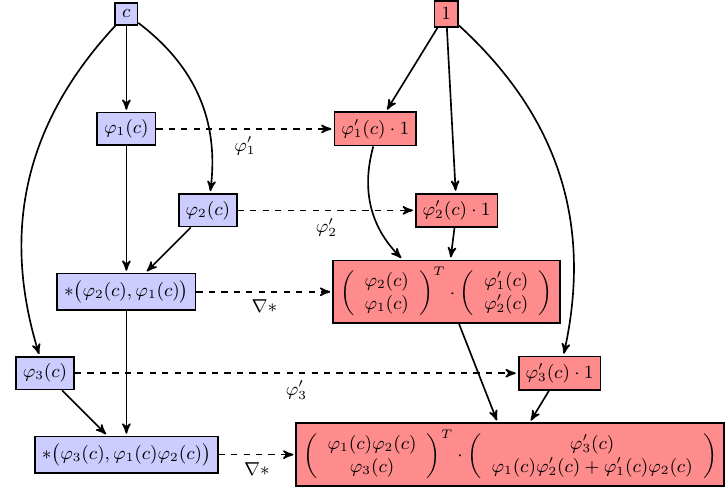}
\caption{Computational graph for the computation of $f'(c) = \frac{d}{dx} \left( \varphi_n * \cdots * \varphi_1 \right)(c)$ in the case $n=3$.}\label{graph for example product rules}
\end{figure}

A further advantage of Forward AD, at least in many implementation, is the efficient handling of common intermediate expressions feeding into several subsequent intermediates. Consider for this the case of a sum of the $n$ elementary uni-variate functions $\varphi_i$ each composed with another univariate elementary function $\psi: \R \to \R$:
\[f = (\varphi_1 \circ \psi) + \cdots + (\varphi_n \circ \psi)= + (\varphi_n \circ \psi, \ + (\varphi_{n-1} \circ \psi, 
+(\cdots (+(\varphi_2 \circ \psi,\varphi_1 \circ \psi)))))\]  
The derivative $f'$ is in this case obviously given by
\[f'(x) = \psi'(x) \cdot \varphi'_1(\psi(x)) + \cdots + \psi'(x) \cdot \varphi'_n(\psi(x)), \]
for all $x \in \R$. To determine $f'(c)$, a symbolic differentiation system will evaluate $\psi(c)$ in each summand, that is $n$ times  (and possibly, if the expression $\psi'(x)$ is not factored out, $\psi'(c)$ as well $n$-times). Hence, the computational cost of evaluating $f'(c)$ is at least $2n+1$ (if $\psi'(x)$ is factored out, $3n$ otherwise), where we again ignore costs for additions and multiplications and for determining the derivative function $f'$. If we want to obtain the value $f(c)$, too, the cost increases to $4n+1$ (or $6n$, respectively). 

In this particular case, some realisations of a Forward Automatic Differentiation system might evaluate $\psi'(c)$ $n$-times as well. However, in many implementations the value $\psi(c)$ will be assigned to a new variable $z$, such that
\begin{align*}
f(c) = \varphi_1(z) + \cdots + \varphi_n(z) \ \ \textrm{and} \ \ f'(c) = z' \cdot \varphi'_1(z) + \cdots + z' \cdot \varphi'_n(z),
\end{align*} 
where $z':=\psi'(c)$. The Forward AD system will then compute pairs starting with 
\begin{align*}
&\big( z, \ z' \cdot 1 \big), \ \ \big( \varphi_1(z),\ \varphi'_1(z) \cdot z' \big), \ \ \big( \varphi_2(z),\ \varphi'_2(z) \cdot z' \big),\\
&\big( \varphi_1(z) + \varphi_2(z),\ \varphi'_1(z) \cdot z' + \varphi'_2(z) \cdot z' \big) \ \ \textrm{etc.}
\end{align*}
Clearly, in this process the values (numbers) $\psi(c)=z$ and $\psi'(c)=z'$ are computed only once. If we again ignore costs for additions and multiplications (including the cost for computing pairs created by lifting $+$), the total cost of determining both $f(c)$ and $f'(c)$ is the cost of computing the $n+1$ pairs $\big( z, \ z' \cdot 1 \big)$,
$\big( \varphi_i(z),\ \varphi'_i(z) \cdot z' \big)$, $i=1,...,n$, which is $2n+2$.\footnote{Admittedly, in this particular case, if one wants to determine only the derivative $f'(c)$, the symbolic evaluation appears to be slightly faster than FAD. However, we have chosen this example mainly to demonstrate how the redundant computation of common sub-expressions can be avoided using Automatic Differentiation. Note further that we have disregarded the cost for determining the derivative function $f'$ in a symbolic differentiation in our considerations.} 
\begin{figure}
\centering
\includegraphics{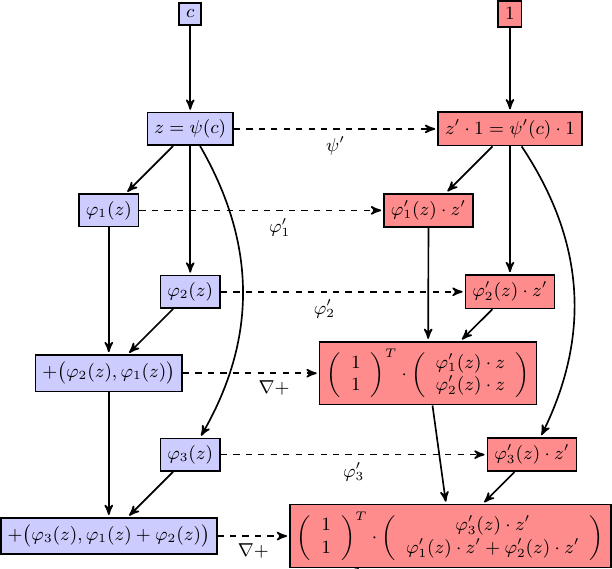}
\caption{Computational graph for the computation of $f(c) = \varphi_1(z) + \cdots + \varphi_n(z)$ and
$f'(c) = z' \cdot \varphi'_1(z) + \cdots + z' \cdot \varphi'_n(z)$ in the case $n=3$, where $z:=\psi(c)$ and $z':=\psi'(c)$.}\label{graph for example sum with common subexpressions}
\end{figure}
Note that the substitution $z:=\psi(x)$ in a symbolic differentiation system would not have the same effect of avoiding redundant calculations. Since $\psi(x)$ is not a number, but an algebraic expression, to obtain $f'(c)$ the variable $x$ would still have to be substituted by $c$ in each instance of $z$ in
$f'(x) = z' \cdot \varphi'_1(z) + \cdots + z' \cdot \varphi'_n(z)$.

Of course, the situation is more difficult when the function $f$ is more complicated or when multi-variate elementary functions other than 
$+$ or $*$ are involved. However, we hope to have demonstrated that, in general, Forward Automatic Differentiation avoids redundant computations of common sub-expressions and does not suffer from the same complexity issues as symbolic computation.

\section{Forward AD and Taylor Series expansion }\label{section Dual and Taylor}

In the literature (see, for example, \cite[Section 2]{pearlmutter-siskind-popl-2007}), definition \eqref{definition extension dual numbers} is sometimes described as being obtained by evaluating the Taylor series expansion of $\widehat{h}$ about $(x_1+x_1' \varepsilon, ...,x_{n}+x'_{n}\varepsilon )$. 

To understand this argument, recall that the Taylor series of an infinitely many times differentiable multivariate function $h: X \to \R$ on an open set $X \subset \R^{n}$ about some point $\mathbf{c}=(c_1,...,c_{n})\in X$ is given by
\begin{align*}
T(h;\mathbf{c})(\mathbf{x}) = \sum_{k_1 + \cdots + k_{n} =0}^{\infty} \frac{(x_1-c_1)^{k_1} \cdots (x_{n}-c_{n})^{k_{n}}}{k_1! \cdots k_{n}!}
\frac{\partial^{k_1+\cdots + k_{n}}h}{\partial x_1^{k_1} \cdots \partial x_{n}^{k_{n}}}(\mathbf{c}),
\end{align*}
for all $\mathbf{x}=(x_1,...,x_{n})\in X$.

Let now $\tilde{h}: \D^{n} \supset X \times \R^{n} \to \D$ be an extension of $h$ to the dual numbers (that is $\tilde{h}|_{X}=h$). We define the Taylor series of $\tilde{h}$ about some vector of dual numbers
$(\mathbf{c},\mathbf{\overset{\rightharpoonup}{c}}):=(c_1 + c'_1 \varepsilon,...,c_{n}+c'_{n} \varepsilon)\in X \times \R^{n}$ analogously to the real case. That is,
\begin{align*}
T(\tilde{h};(\mathbf{c},\mathbf{\overset{\rightharpoonup}{c}}))((\mathbf{x},\mathbf{\overset{\rightharpoonup}{x}})&)\\[1ex] 
= \sum_{k_1 + \cdots + k_{n} =0}^{\infty}\Bigg(&\frac{(x_1 - c_1 + (x'_1 - c'_1) \varepsilon)^{k_1} 
\cdots (x_{n} -c_{n} +( x'_{n} - c'_{n}) \varepsilon)^{k_{n}}}{k_1! \cdots k_{n}!}\\
&\cdot \frac{\partial^{k_1+\cdots + k_{n}}\tilde{h}}{\partial x_1^{k_1} \cdots \partial x_{n}^{k_{n}}}((\mathbf{c},\mathbf{\overset{\rightharpoonup}{c}}))\Bigg),
\end{align*}
for all $(\mathbf{x},\mathbf{\overset{\rightharpoonup}{x}}):=(x_1+x'_1 \varepsilon,...,x_{n}+ x'_{n} \varepsilon)\in X \times \R^{n}$.

Trivially, this series converges for $(\mathbf{x},\mathbf{\overset{\rightharpoonup}{x}}) = (\mathbf{c},\mathbf{\overset{\rightharpoonup}{c}})$. Further, due to
$\varepsilon^2 = 0$, the Taylor series about any $(\mathbf{x},\mathbf{0})=(x_1+0\varepsilon,...,x_{n}+0\varepsilon)\in X \times \R^{n}$ converges for the arguments $(\mathbf{x},\mathbf{\overset{\rightharpoonup}{x}}) =(x_1 + x'_1 \varepsilon ,...,x_{n}+x'_{n}\varepsilon )$, for all 
$\mathbf{\overset{\rightharpoonup}{x}} \in \R^{n}$. We have, identifying $\mathbf{x}$ with $(\mathbf{x},\mathbf{0})$, 
\begin{align}\label{dual numbers Taylor series expansion}\notag
T(\tilde{h};\mathbf{x})(\mathbf{x},\mathbf{\overset{\rightharpoonup}{x}}) 
&= \sum_{k_1 + \cdots + k_{n} =0}^{\infty} \frac{(x'_1\varepsilon)^{k_1} 
\cdots (x'_{n}\varepsilon)^{k_{n}}}{k_1! \cdots k_{n}!}
\frac{\partial^{k_1+\cdots + k_{n}}}{\partial x_1^{k_1} \cdots \partial x_{n}^{k_{n}}}\tilde{h}(\mathbf{x})\\[1ex] 
&=\sum_{k_1 + \cdots + k_{n} =0}^{1}\frac{(x'_1\varepsilon)^{k_1} 
\cdots (x'_{n}\varepsilon)^{k_{n}}}{k_1! \cdots k_{n}!}
\frac{\partial^{k_1+\cdots + k_{n}}}{\partial x_1^{k_1} \cdots \partial x_{n}^{k_{n}}}\tilde{h}(\mathbf{x}) \\[1ex]\notag
&=\tilde{h}(\mathbf{x}) + \sum_{j=1}^{n}\frac{\partial}{\partial x_j}\tilde{h}(\mathbf{x}) \cdot x'_j\varepsilon\\[1ex] \notag
&=\tilde{h}(\mathbf{x}) + \left(\nabla \tilde{h}(\mathbf{x}) \cdot \left(\begin{array}{c} x'_1 \\ \vdots \\ x'_{n} \end{array} \right) \right) \varepsilon% \\[1ex] \notag
=h(\mathbf{x}) + (\nabla h(\mathbf{x}) \cdot \mathbf{\overset{\rightharpoonup}{x}}) \varepsilon,   
\end{align}
where $\nabla \tilde{h}(\mathbf{x}):=\left(\frac{\partial \tilde{h}}{\partial x_1}(\mathbf{x}) \cdots \frac{\partial \tilde{h}}{\partial x_n}(\mathbf{x})\right) = \nabla h(\mathbf{x})$.

As we see, the right-hand side of the last equation is equal to\\ $\widehat{h}(x_1+x'_1\varepsilon,...,x_{n}+x'_{n}\varepsilon)$ in definition \eqref{definition extension dual numbers}. Hence, if we choose $\tilde{h}$ as $\widehat{h}$, we obtain 
\begin{align}\label{Taylor series equal extension to dual numbers} \widehat{h}(x_1+x'_1\varepsilon,...,x_{n}+x'_{n}\varepsilon) = T(\widehat{h};(x_1,...,x_{n}))(x_1+x'_1\varepsilon,...,x_{n}+x'_{n}\varepsilon). 
\end{align}
That is: 

\begin{proposition}
The extension of an infinitely many times differentiable $h: X \to \R$, on open $X \subset \R^{n}$, to a set $X \times \R^n \subset \D^n$ as defined in \eqref{definition extension dual numbers}, is the (unique) function $\widehat{h}$, with the property that the images of any $(x_1+x'_1\varepsilon,...,x_{n}+x'_{n}\varepsilon) \in X \times \R^n$ under $\widehat{h}$ and $T(\widehat{h};(x_1,...,x_{n}))$ are equal.
\end{proposition}

It is clear that the statement remains true if we replace `infinitely many times differentiable' by `differentiable' and $T(\widehat{h};(x_1,...,x_{n}))$ by the first-degree Taylor polynomial $T_1(\widehat{h};(x_1,...,x_{n}))$.\footnote{Note that we mean here by first-degree Taylor polynomial simply the polynomial in \eqref{dual numbers Taylor series expansion}. That is, we make no statement about the existence of a remainder term, for which mere differentiability of $h$ would not be sufficient.}

Since we identify $X$ with its natural embedding into $\D^{n}$, we can replace $\tilde{h}(\mathbf{x})$ by $h(\mathbf{x})$ in the right-hand side of \eqref{dual numbers Taylor series expansion}. It is custom to do this in the left-hand side of \eqref{dual numbers Taylor series expansion} as well. That is, one 
usually writes $T({h};(x_1,...,x_{n}))$ instead of $T(\tilde{h};(x_1,...,x_{n}))$ or $T(\widehat{h};(x_1,...,x_{n}))$.

By \eqref{Taylor series equal extension to dual numbers}, it is obvious that one can describe the process of determining the directional derivatives 
$\nabla f_j(\mathbf{c}) \cdot \mathbf{\overset{\rightharpoonup}{x}}$ of each $f_j$ in terms of Taylor series expansion, if $f_j$ is infinitely many times differentiable, or its first-degree Taylor polynomial, otherwise. 

\begin{comment}
\begin{proposition}
Given $\mathbf{c} \in X \subset \R^n$ and $\mathbf{\overset{\rightharpoonup}{x}} \in \R^n$, the evaluation of $J_f(\mathbf{c}) \cdot \mathbf{\overset{\rightharpoonup}{x}}$ of an automatically differentiable function $f: X \to \R^m$ can be achieved through the lifting of each $f_j$ to a function on $ X \times \R^n$ and evaluating its first-degree Taylor polynomial $T_1(f_j;\mathbf{c})(\mathbf{c},\mathbf{\overset{\rightharpoonup}{x}})$ about $\mathbf{c}=(\mathbf{c},\mathbf{0})$ at $(\mathbf{c},\mathbf{\overset{\rightharpoonup}{x}})$ as 
given in \eqref{dual numbers Taylor series expansion}. This process is equivalent to the computations of the first-degree Taylor polynomials of the elementary functions $\varphi_i$ about suitable $\mathbf{c_i}$ at suitable $(\mathbf{c_i},\mathbf{\overset{\rightharpoonup}{x_i}})$, which is 
equivalent to the computation of the pairs \eqref{pairs elementary functions in dual numbers Section} in suitable order.
\end{proposition}
\end{comment}

Taylor series expansion or Taylor polynomials can also be used to compute higher-order partial derivatives which we discuss briefly in Section \ref{Higher-Order Partial} (see also the work in \cite[Chapter 13]{Griewank2008}).

\section{Forward AD, Differential Geometry and Category Theory}\label{section diffgeom}

In recent literature (see \cite{Manzyuk-mfps2012}) the extension of differentiable functions $h: X \to \R$ on open $X \subset \R^{n}$ to a function 
$\widehat{h}: \D^{n} \supset X \times \R^{n} \to \D$ is described in terms of the \emph{push-forward} operator known from Differential Geometry. We shortly summarize the discussion provided in \cite{Manzyuk-mfps2012}.

Let $M,N$ be differentiable manifolds, $TM, TN$ their tangent bundles and let $h: M \to N$ be a differentiable function. The push-forward (or \emph{differential}) $T(h)$ of $h$ can be defined\footnote{The definition we are using here is the same as in \cite{Manzyuk-mfps2012}. Some authors define $T(h)$ via $T(h)(\mathbf{x},\mathbf{\overset{\rightharpoonup}{x}}) := d_{\mathbf{x}}h(\mathbf{\overset{\rightharpoonup}{x}})$.} as 
\[T(h): TM \to TN \ \ \ \textrm{with} \ \ T(h)(\mathbf{x},\mathbf{\overset{\rightharpoonup}{x}}) = (h(\mathbf{x}),d_{\mathbf{x}}h(\mathbf{\overset{\rightharpoonup}{x}})),\]
where $d_{\mathbf{x}}h(\mathbf{\overset{\rightharpoonup}{x}})$ is the \emph{push-forward (or differential) of $h$ at $\mathbf{x}$} applied to 
$\mathbf{\overset{\rightharpoonup}{x}}$.

If now $f: X \to \R^m$ on open $X \subset \R^{n}$ is a differentiable function, this reads  
\[T(f): X \times \R^{n} \to \R^m \times \R^m \ \ \ \textrm{with} \ \ T(f)(\mathbf{x},\mathbf{\overset{\rightharpoonup}{x}}) = \left(f(\mathbf{x}),\ J_f(\mathbf{x}) \cdot \mathbf{\overset{\rightharpoonup}{x}} \right)
.\] 
Considering $X \times \R^{n}$ as a subset of $\D^n$ and identifying $\R^m \times \R^m$ with $\D^m$, in light of \eqref{extension dual numbers mulit-dimensional}, this means nothing else than 
\[T(f)= \widehat{f}.\] 
Furthermore, in the special case of a real-valued and infinitely many times differentiable function $f_j: X \to \R$ on open $X \subset \R^{n}$ we also have, by equation \eqref{Taylor series equal extension to dual numbers},
\[T(f_j)(\mathbf{x},\mathbf{\overset{\rightharpoonup}{x}}) = T(f_j;\mathbf{x})(\mathbf{x},\mathbf{\overset{\rightharpoonup}{x}}), \ \ \forall (\mathbf{x},\mathbf{\overset{\rightharpoonup}{x}}) \in X \times \R^{n},\] which justifies using the letter $T$ for both, the push-forward and the Taylor-series of $f_j$ in this setting. 

It is well-known that $T(id_M) = id_{TM}$ and that
\[T(h_2 \circ h_1) = T(h_2) \circ T(h_1),\]
for all $h_1: M \to N$ and $h_2: N \to L$, for differentiable manifolds $M,N,L$. That is, the mapping given by
\begin{align*}
&M \mapsto TM \\
&h \mapsto T(h)
\end{align*} 
is a functor from the category of differentiable manifolds to the category of vector bundles (see, for example, \cite[III, \S 2]{Lang}). Furthermore, since $T\R = \R^2$, one can even consider the algebra of dual number $\D$ as the image of $\R$ under $T$, equipped with the push-forwards of addition and multiplication. I.e.,
\[(\D,+,\cdot) = \left(T\R,T(+),T(*)\right).\]
Extending this to higher dimensions, the lifting of a differentiable function\\ 
$f:X \to \R^m$ on open $X \subset \R^{n}$ to a function $\widehat{f}$ on a set $X \times \R^{n} \subset \D^{n}$ may be considered as the application of the functor $T$ to $X$, $\R^m$ and $f$. In other words,
\begin{align*}
 \widehat{f}: X \times \R^{n} \to \D^m \ \ \ = \ \ \ T(f):TX \to T\R^m \ \ \ = \ \ \ T\left(  f: X \to \R^m \right).
\end{align*}
In summary, the Forward Mode of AD may also be studied from viewpoints of Differential Geometry and Category Theory. This fact may be used to generalise the concept of Forward AD to functions operating on differentiable manifolds other than open subsets of $\R^n$.

\section{Higher-Order Partial Differentiation}\label{Higher-Order Partial}

\subsection{Truncated Polynomial Algebras}\label{polynomial algebra}

The computation of higher-order partial derivatives of a sufficiently often and automatically differentiable function $f_j: U \to \R$ on open $U \subset \R^{n}$ can, for example, be achieved through the extension of $f_j$ to a function defined on a truncated polynomial algebra.\footnote{We denote the domain of definition by $U$ here to avoid confusion with indeterminates which we denote by $X$ or $X_i$.} This approach has, for instance, been described by Berz\footnote{Berz actually uses a slightly different but equivalent approach than the one presented here (see the end of this subsection).} in \cite{Berz} with further elaborations to be found in \cite{Garczynski} and the work in \cite{KarczmarczukI} and \cite{pearlmutter-siskind-popl-2007} extending the idea (for the two latter, see the following subsection). 

Indeed, the extension of $f_j$ to a function on dual numbers as given in definition \eqref{definition extension dual numbers} can already be considered in the context of truncated polynomials, since $\D \cong \R[X]/(X^2)$. 

Let now $\N \ni n, N \geq 1$ and consider the algebra $\R[X_1,...,X_n]/I_{N}$, where 
\[I_{N}:=\left(\left\{X_1^{k_1} \cdots X_n^{k_n} \ | \ (k_1,...,k_n) \in \N^n \ \textrm{with} \ k_1 + \cdots + k_n > N\right\}\right)\]
is the ideal generated by all monomials of order $N+1$. Then 
\[\R[X_1,...,X_n]/I_{N} \cong \left\{\sum_{k_1 + \cdots + k_{n} =0}^{N}x_{(k_1,...,k_n)}X_1^{k_1} \cdots X_n^{k_n} \ | \ x_{(k_1,...,k_n)} \in \R\right\}\]
consists of all polynomials\footnote{Which, of course, can be identified with tuples consisting of their coefficients.} in $X_1,...,X_n$ of degree $\leq N$.

Denote 
\[\f_i := x_{i,(0,...,0)} + \sum_{k_1 + \cdots + k_{n} =1}^{N}x_{i,(k_1,...,k_n)}X_1^{k_1} \cdots X_n^{k_n} \]
and, for simplicity, identify $x_i:=x_{i,(0,...,0)}$. 
Let further $U \subset \R^n$ be open and define
\begin{align*}
&\left(\R[X_1,...,X_n]/I_{N}\right)^n_U\\
&:=\left\{(\f_1,...,\f_n) \in \left(\R[X_1,...,X_n]/I_{N}\right)^n \ |\ (x_1,...,x_n) \in U \right\}.
\end{align*}
That is, $\left(\R[X_1,...,X_n]/I_{N}\right)^n_U$ consists of vectors of polynomials in\\
$\R[X_1,...,X_n]/I_{N}$ with the property that the vector consisting of the trailing coefficients lies in $U$.

We now define an extension of an $N$-times differentiable, real-valued function
$h: U \to \R$ to a function
$\widehat{\widehat{h}}: \left(\R[X_1,...,X_n]/I_{N}\right)^n_U \to
\R[X_1,...,X_n]/I_{N}$ via
\begin{align}\label{definition truncated polynomials}\notag
&\widehat{\widehat{h}}(\f_1,...,\f_n) \\
&:= \sum_{k_1 + \cdots + k_{n} =0}^{N} \frac{\left(\f_1-x_{1}\right)^{k_1} 
\cdots \left(\f_n-x_{n}\right)^{k_{n}}}{k_1! \cdots k_{n}!}
\frac{\partial^{k_1+\cdots + k_{n}}h}{\partial x_1^{k_1} \cdots \partial x_{n}^{k_{n}}}(x_{1},...,x_{n})
\end{align}
In other words, $\widehat{\widehat{h}}$ is the unique function with the property that the images of any $(\f_1,...,\f_n)$ under 
$\widehat{\widehat{h}}$ and its $N$-th degree Taylor polynomial \\
$T_N(\widehat{\widehat{{h}}};(x_1,...,x_n))$ about $(x_1,...,x_n)$ are equal.

The reason for this definition becomes apparent when we apply $\widehat{\widehat{h}}$ to a vector of polynomials of the form $(x_1+X_1, ...,x_{n}+X_n )$. Then
\begin{align*}
&\widehat{\widehat{h}}(x_1+X_1, ...,x_{n}+X_n )\\ \notag
&= h(\mathbf{x}) 
+ \sum_{k_1 + \cdots + k_{n} =1}^{N}\frac{1}{k_1!\cdots k_n!}X_1^{k_1} \cdots X_n^{k_n}
\frac{\partial^{k_1+\cdots + k_{n}}h}{\partial x_1^{k_1} \cdots \partial x_{n}^{k_{n}}}(\mathbf{x}),
\end{align*}
for $\mathbf{x}=(x_1,...,x_n)$. 

One can now compute partial derivatives of order $N$ of a sufficiently often and automatically differentiable function $f_j: U \to \R$ by implementing \eqref{definition truncated polynomials} for all elementary functions $\varphi_i$. This leads to the extension of $f_j$ to 
$\widehat{\widehat{f_j}}$ and one obtains a partial derivative $\frac{\partial^{k_1+\cdots + k_{n}}f_j}{\partial x_1^{k_1} \cdots \partial x_{n}^{k_{n}}}(\mathbf{c})$ at a given $\mathbf{c}=(c_1,...,c_n) \in U$ as the ($k_1! \cdots k_n!$)-th multiple of the coefficient of $X_1^{k_1} \cdots X_n^{k_n}$ in\\
$\widehat{\widehat{f_j}}(c_1+X_1, ...,c_{n}+X_n )$. 

Of course, to show that this method actually works, one needs to prove an analogue of Proposition \ref{proposition dual numbers quasi-compositions}. However, we will omit the proof here.

Obviously, this method requires the computation of the $N$-th Taylor polynomial, or, equivalently, of the Taylor coefficients up to degree $N$, of each elementary function $\varphi_i$ appearing in $f_j$. The complexity of this problem is discussed in detail in \cite[pages 306--308]{Griewank2008}: If no restrictions on an elementary function $\varphi_i$ is given, even in the uni-variate case, order-$N^3$ arithmetic operations may be required. However, in practice all elementary functions $\varphi_i$ are solutions of linear ODEs which reduces the computational costs of their Taylor coefficients to $k N^2+\mathcal{O}(N)$ for $k \in \{1,2,3\}$ (see \cite[(13.7) and Proposition 13.1]{Griewank2008}).

We further remark that Berz in \cite{Berz}, instead of $\R[X_1,...,X_n]/I_{N}$, actually uses the algebra 
\[\left\{\sum_{k_1 + \cdots + k_{n} =0}^{N}x_{(k_1,...,k_n)}\frac{X_1^{k_1}}{k_1!} \cdots \frac{X_n^{k_n}}{k_n!} \ | \ x_{(k_1,...,k_n)} \in \R\right\},\] where multiplication is defined via
\begin{align}\label{multiplication Berz}
&\left(\sum_{r_1 + \cdots + r_{n} =0}^{N}x_{(r_1,...,r_n)}\frac{X_1^{r_1}}{r_1!} \cdots \frac{X_n^{r_n}}{r_n!} \right)
\cdot \left(\sum_{s_1 + \cdots + s_{n} =0}^{N}y_{(s_1,...,s_n)}\frac{X_1^{s_1}}{s_1!} \cdots \frac{X_n^{s_n}}{s_n!} \right)\\
&= \sum_{k_1 + \cdots + k_{n} =0}^{N} (k_1 + \cdots + k_n)! \cdot z_{(k_1,...,k_n)}\frac{X_1^{k_1}}{k_1!} \cdots \frac{X_n^{k_n}}{k_n!}. 
\end{align}
for $z_{(k_1,...,k_n)}:=\left(\sum_{(r_1,...,r_n)+(s_1,...,s_n)=(k_1,...,k_n)}x_{(r_1,...,r_n)}y_{(s_1,...,s_n)}\right)$. This obviously makes no real difference to the theory, the main advantage is that after lifting a function $h$ to this algebra, one can extract partial derivatives directly, without the need to multiply with $k_1! \cdots k_n!$.

%We only remark that Pearlmutter and Siskind in \cite{pearlmutter-siskind-popl-2007} describe an extension of this idea, which allows for the evaluation of partial derivatives of arbitrary degree of an infinitely many times differentiable function $h$. In simple terms, they extend $h$ to a function on a \underline{non-truncated} polynomial algebra $\left(\R[X_1,...,X_n]\right)^n_U$ through successively increasing the degree of truncation $N$ \emph{ad infinitum} in a \emph{lazy evaluation}.

\subsection{Differential Algebra and Lazy Evaluation}

Differential algebra is an area which has originally been developed to provide algebraic tools for the study of differential equations (see, for example, the original work by  Ritt \cite{Ritt} or the introductory article \cite{HuLu}). In the context of Automatic Differentiation it was utilized in \cite{KarczmarczukI}, \cite{KarczmarczukII}. 

A \emph{differential algebra} is an algebra $\mathcal{A}$ with a mapping $\delta: \A \to \A$ called a \emph{derivation}, which satisfies
\begin{align*}
\delta(a+b) = \delta(a) + \delta(b) \ \ \ \textrm{and} \ \ \ \delta(a \cdot b) = \delta(a) \cdot b + a \cdot \delta(b),
\end{align*}
for all $a,b \in \A$. If $\A$ is a field, it is called a \emph{differential field}. Examples for differential fields or algebras are the 
set of (real or complex) rational functions with any partial differential operator or polynomial algebras $\R[X_1,...,X_n]$ with a formal partial derivative.%\footnote{In the context of the latter, one could also refer to the work in \cite{pearlmutter-siskind-popl-2007} as being based on differential algebra.} 
Truncated polynomial algebras as described in the previous section can be made into differential algebras as well, however, as Garczynski in \cite{Garczynski} points out, a formal (partial) derivative is not a derivation in that case. (For instance, the mapping 
$D\cdot X: \R[X]/(X^2) \to \R[X]/(X^2)$ with $(D\cdot X)(x+x'X) = x'X$ is a derivation, but the formal derivative $D$ with $D(x+x'X) = x'$ is not.)

Karczmarczuk describes now in \cite{KarczmarczukI} a system in which, through a \emph{lazy evaluation}, to each object $a \in \A$ in a differential field $\A$, the sequence 
\[\left(\delta^n(a)\right)_{n \in \N} = (a, \delta(a), \delta^2(a), \delta^3(a),...) \]
is assigned. In the case of an infinitely many times differentiable univariate function $h:J \to \R$, on open $J \subset \R$, and the differential operator as derivation, this obviously gives $(h, h', h'', h''',...)$.
 
The set $\{\left(\delta^n (a)\right)_{n \in \N} \ | \ a \in \A\}$ now forms a differential algebra itself, where addition is defined entry-wise, multiplication is given by
\begin{align}\label{eq. multiplication diff-algebra} \notag
&\left(\delta^n(a)\right)_{n \in \N} \cdot \left(\delta^n(b)\right)_{n \in \N}
:= \left(\delta^n(a\cdot b)\right)_{n \in \N}\\[1ex] \notag 
&=\left(a \cdot b, a \cdot \delta(b)+\delta(a) \cdot b, a\cdot \delta^2(b) + 2 ( \delta(a) \cdot \delta(b) ) + \delta^2(a) \cdot b,\right.\\ \notag
&\left. \ \ \ \ \ a \cdot \delta^3(b) + 3 (\delta^2(a) \cdot \delta(b)) + 3 ( \delta(a) \cdot \delta^2(b) ) + \delta^3(a) \cdot b,...\right)
\\[1ex]
&= \left( \sum_{k_a+k_b=n}\frac{n!}{k_a!k_b!}\delta^{k_a}(a)\cdot\delta^{k_b}(b) \right)_{n \in \N}
\end{align}
and the derivation, denoted by $df$, is the right-shift operator, given by
\[df(a, \delta(a), \delta^2(a), \delta^3(a),...) = (\delta(a), \delta^2(a), \delta^3(a), \delta^4(a),...),\]
for all $a,b \in \A$. (These definition are given recursively in the original work; for implementation details, see \cite[Subsections 3.2--3.3]{KarczmarczukI}.)

To utilize these ideas for Automatic Differentiation, the assignment\\
$\varphi_i \mapsto \left(\delta^n(\varphi_i)\right)_{n \in \N}$ is implemented for all uni-variate infinitely many times 
differentiable elementary functions $\varphi_i: J_i \to \R$, defined on open $J_i \subset \R$, and the elementary functions $+$, $*$ and $/$ are replaced by addition, multiplication and division\footnote{We omit the description of how division on $\{\left(\delta^n (a)\right)_{n \in \N} \ | \ a \in \A\}$ is defined here, which can be found in the original publication.} on the $\left(\delta^n(\varphi_i)\right)_{n \in \N}$. This then generates for any univariate automatically differentiable function $f: J \to \R$, on open $J \subset \R$, which is constructable by the $\varphi_i$, $+$, $*$ and $/$, the sequence $(f, f', f'', f''',...)$. The $N$-the derivative of $f$ is then the first entry (which is distinguished in the implementation in \cite{KarczmarczukI}) in $df^N(f, f', f'', f''',...)$.

However, since a differential algebra/field is an abstract concept, this approach can, in principle, be applied to other objects than differentiable functions

We only remark that Karczmarczuk briefly describes in \cite{KarczmarczukIII} a generalisation to the multi-variate case. Kalman in \cite{kalman-2002a} also constructs a system which appears to be similar. 

Comparing \eqref{eq. multiplication diff-algebra} with \eqref{multiplication Berz} shows that the multiplication on\\ 
$\{\left(\delta^n (a)\right)_{n \in \N} \ | \ a \in \A\}$ is identical to the multiplication on the algebra used by Berz. Hence, one can express the described system, at least in the case of $a=h$ being a differentiable function, in terms of polynomial algebras. That is, consider the truncated algebra 
$\left\{\sum_{k =0}^{N}x_k\frac{X^{k}}{k!}  \ | \ x_{k} \in \R\right\}$ with multiplication as in \eqref{multiplication Berz}
and define an extension of a differentiable $h: J \to \R$ to a mapping 
\[\widehat{\tilde{h}}: \left\{\f=x + \sum_{k = 1}^N x_k \frac{X^{k}}{k!}  \ | \ x \in J, x_k \in \R\right\} \to \left\{\sum_{k =0}^{N}x_k\frac{X^{k}}{k!}  \ | \ x_{k} \in \R\right\} \] via
\begin{align*}
\widehat{\tilde{h}}(\f):= \sum_{k =0}^{N} \left( \f -x \right)^{k}\frac{1}{k!}\frac{d^k h}{{dx}^k}(x), 
\end{align*}
for all $\f \in \left\{\f=x + \sum_{k = 1}^N x_k \frac{X^{k}}{k!}  \ | \ x \in J, x_k \in \R\right\}$.
Then, for all $x \in J$,
\begin{align*}
\widehat{\tilde{h}}(x+X) = h(x) + \sum_{k=1}^{N} \frac{X^{k}}{k!} \frac{d^k h}{{dx}^k}(x),
\end{align*}
which corresponds to the tuple $(h(x),h'(x),h''(x),...,h^{(N)}(x))$. The lazy evaluation technique in \cite{KarczmarczukI} increases the degree of truncation $N$ successively \emph{ad infinitum} to compute any entry of the sequence $(h(x),h'(x),h''(x),h'''(x),...)$ for any $x \in J$.

Similarly, Pearlmutter and Siskind describe in \cite{pearlmutter-siskind-popl-2007} the lifting of a multi-variate function $h:U \to \R$ to a function on $\left(\R[X_1,...,X_n]/I_{N}\right)^n_U$ as defined in \eqref{definition truncated polynomials}, and then, through a lazy evaluation, increase the degree of truncation $N$ successively.\footnote{Pearlmutter and Siskind actually seem to use the ideal 
$(X_1^{N+1},...,X_n^{N+1})$ instead of $I_N$, which makes no real difference.} This computes any entry of the sequence\\ 
$\left(\frac{\partial^{k_1+\cdots + k_{n}}h}{\partial x_1^{k_1} \cdots \partial x_{n}^{k_{n}}}(\mathbf{x})\right)_{(k_1,...,k_n) \in \N^n}$ (given in some order), for any vector $\mathbf{x} \in U$.

\section{The Reverse Mode of AD}\label{section Reverse}

%\subsection{An elementary approach to Reverse AD}
Let, as before, $f:X \to \R^m$ on open $X \subset \R^n$ be automatically differentiable. 
As already mentioned, the Reverse Mode of Automatic Differentiation evaluates products of the Jacobian of $f$ with row vectors. That is, it computes
\[ \mathbf{\overset{\leftharpoonup}{y}}  \cdot J_f(\mathbf{c}), \ \ \ \  \textrm{for fixed} \ \mathbf{c}\in X \ \textrm{and} \ \mathbf{\overset{\leftharpoonup}{y}} \in \R^{1 \times m}.\] 
To our knowledge, there exists currently no method to achieve this computation, which resembles Forward AD using dual numbers. Instead, an elementary approach, similar to the Forward AD approach in Section \ref{section Griewank} will have to suffice. The Reverse Mode is, for example, described in \cite{Griewank2003AMV}, \cite{Griewank2008} and \cite{PearlmutterSiskind2008a}. We follow mainly the discussion in \cite{Griewank2003AMV}.

Express again $f$ as the composition $P_Y \circ \Phi_{\mu} \circ \cdots \circ \Phi_1 \circ P_X$,
with $P_X$, $P_Y$ and the $\Phi_i$ as in Section \ref{section Griewank}. The computation of $\mathbf{\overset{\leftharpoonup}{y}}  \cdot J_f(\mathbf{c})$ is, by the chain rule, the evaluation of the product
\begin{align}\label{matrix-vector product Reverse Griewank}\notag
\mathbf{\overset{\leftharpoonup}{y}} \cdot J_f(\mathbf{c}) &= 
\mathbf{\overset{\leftharpoonup}{y}} \cdot P_Y \cdot \Phi'_{\mu,\mathbf{c}}\ \cdots \ \Phi'_{1,\mathbf{c}} \cdot P_X \\[1ex]
\Leftrightarrow \ \ \ \ J_f(\mathbf{c})^T \cdot \mathbf{\overset{\leftharpoonup}{y}}^T &= P_X^T \cdot \Phi'^T_{1,\mathbf{c}} 
\ \cdots \ \Phi'^T_{\mu,\mathbf{c}} \cdot P_Y^T \cdot \mathbf{\overset{\leftharpoonup}{y}}^T,
\end{align}
where again $\Phi'_{i,\mathbf{c}}$ denotes the Jacobian of $\Phi_i$ at $\left(\Phi_{i-1}\circ \cdots \circ \Phi_1 \circ P_X \right)(\mathbf{c})$. 

Obviously, the sequence of state vectors $\mathbf{\mathbf{v}}^{[i]} \in H=\R^{n+\mu}$ is the same as in the Forward Mode case. (Here, $\mu$ is, of course, again the total number of elementary functions which make up the function $f$.) The difference lies in the computation of the evaluation trace of \eqref{matrix-vector product Reverse Griewank}, which we denote by
$\mathbf{\mathbf{\overline{v}}}^{[\mu]} = \mathbf{\mathbf{\overline{v}}}^{[\mu]}(\mathbf{c},\mathbf{\overset{\leftharpoonup}{y}}),..., \mathbf{\mathbf{\overline{v}}}^{[0]}=\mathbf{\mathbf{\overline{v}}}^{[0]}(\mathbf{c},\mathbf{\overset{\leftharpoonup}{y}})$.

For simplicity, assume $P_Y(v_1,...,v_{n+\mu}) = (v_{n+\mu-m},...,v_{n+\mu})$, for all\\ $(v_1,...,v_{n+\mu}) \in H$ and denote 
$\mathbf{\overset{\leftharpoonup}{y}}^T = (y_1',...,y_m')$. We define the evaluation trace of \eqref{matrix-vector product Reverse Griewank} as 
\[\mathbf{\mathbf{\overline{v}}}^{[\mu]}:= P_Y^T \cdot \mathbf{\overset{\leftharpoonup}{y}}^T = (0,...,0,y_1',...,y_m') \ \ \textrm{and} \ \ \mathbf{\mathbf{\overline{v}}}^{[i-1]}:= \Phi'^T_{i,\mathbf{c}} \cdot \mathbf{\mathbf{\overline{v}}}^{[i]}.\]
By \eqref{Phi'},
\begin{align*} &\ \ \ \ \ \ \ \ \ \ \ \ \ \ \ \ \ \ \ \ \ \ \ \ \ \ (n+i)\textrm{-th column}\\
&\ \ \ \ \ \ \ \ \ \ \ \ \ \ \ \ \ \ \ \ \ \ \ \ \ \ \ \ \ \ \ \ \ \ \downarrow\\[1ex]
&\Phi'^T_{i, \mathbf{c}} = \left(\begin{array}{ccccccc} 
1 & \cdots & 0 & \frac{\partial \varphi_i}{\partial  v_{1}}
(\cdots )  & 0 & \cdots & 0 \\													
\vdots & \ddots & \vdots & \vdots & \vdots & \ & \vdots \\
0 & \cdots & 1 & \vdots & 0 & \cdots & 0 \\
0 & \cdots & 0 & \vdots & 1 & \cdots & 0 \\
\vdots & \ & \vdots & \vdots & \vdots & \ddots & \vdots \\
0 & \cdots & 0 & \ \ \ \frac{\partial \varphi_i}{\partial  v_{n+\mu}}
(\cdots )  & 0 & \cdots & 1 
\end{array}\right),	
\end{align*}
where $\frac{\partial \varphi_i}{\partial  v_{k}}
(\cdots ) = \frac{\partial \varphi_i}{\partial  v_k}(v_{i_1}(\mathbf{c}),...,v_{i_{n_i}}(\mathbf{c}) )$ is interpreted as $0$ if $\varphi_i$ does not depend on $v_k$, and the $v_{i_1}(\mathbf{c}),...,v_{i_{n_i}}(\mathbf{c}) \in \{\mathbf{\mathbf{v}}^{[i-1]}_1(\mathbf{c}),...,\mathbf{\mathbf{v}}^{[i-1]}_{n+i-1}(\mathbf{c})\}$.

Therefore, each $\mathbf{\mathbf{\overline{v}}}^{[i-1]}$ is of the form
\begin{align}\notag \label{v'[i] in reverse} \mathbf{\mathbf{\overline{v}}}^{[i-1]} &=\left(\begin{array}{c} \mathbf{\overline{v}}^{[i]}_{n+i} \cdot \frac{\partial \varphi_i}{\partial  v_1}(\cdots ) + \mathbf{\overline{v}}^{[i]}_{1} \\ \vdots \\  \mathbf{\overline{v}}^{[i]}_{n+i} \cdot \frac{\partial \varphi_i}{\partial  v_{n+i-1}}(\cdots ) + \mathbf{\overline{v}}^{[i]}_{n+i-1} \\[2ex]\mathbf{\overline{v}}^{[i]}_{n+i} \cdot \frac{\partial \varphi_i}{\partial  v_{n+i}}(\cdots ) \\[2ex]  \mathbf{\overline{v}}^{[i]}_{n+i} \cdot \frac{\partial \varphi_i}{\partial  v_{n+i+1}}(\cdots ) + \mathbf{\overline{v}}^{[i]}_{n+i+1}\\ \vdots \\ \mathbf{\overline{v}}^{[i]}_{n+i} \cdot  \frac{\partial \varphi_i}{\partial  v_{n+\mu}}(\cdots ) + \mathbf{\overline{v}}^{[i]}_{n+\mu}  \end{array} \right)\\[1.5ex]
&=  \mathbf{\overline{v}}^{[i]}_{n+i}\  \cdot \ \left(\begin{array}{c} \frac{\partial \varphi_i}{\partial  v_1}(\cdots )  \\ \vdots \\  \frac{\partial \varphi_i}{\partial  v_{n+i-1}}(\cdots ) \\[2ex]\frac{\partial \varphi_i}{\partial  v_{n+i}}(\cdots ) \\[2ex]  \frac{\partial \varphi_i}{\partial  v_{n+i+1}}(\cdots ) \\ \vdots \\ \frac{\partial \varphi_i}{\partial  v_{n+\mu}}(\cdots )  \end{array} \right)
+  
\left(\begin{array}{c} \mathbf{\overline{v}}^{[i]}_{1} \\ \vdots \\  \mathbf{\overline{v}}^{[i]}_{n+i-1} \\ 0 \\ \mathbf{\overline{v}}^{[i]}_{n+i+1}\\ \vdots \\ \mathbf{\overline{v}}^{[i]}_{n+\mu}  \end{array} \right)
\end{align}

The value $\mathbf{\overset{\leftharpoonup}{y}} \cdot J_f(\mathbf{c})$ is then given by
\[\left(\mathbf{\overset{\leftharpoonup}{y}}  \cdot J_f(\mathbf{c})\right)^T = P^T_X \cdot \mathbf{\overline{v}}^{[0]}.\]

If we let $\overline{\varphi}_i: H \to \R$ be an extension of $\varphi_i: U_i \to \R$ to $H$ with $\frac{\partial \overline{\varphi}_i}{\partial v_k} = 0$ if 
$\varphi_i$ does not depend on $v_k$, and define $\mathbf{\mathbf{\overline{v}}}^{[i,*]} \in H$ by 
$\mathbf{\mathbf{\overline{v}}}^{[i,*]}_k = \mathbf{\mathbf{\overline{v}}}^{[i]}_k$, for $k\neq n+i$, and $\mathbf{\mathbf{\overline{v}}}^{[i,*]}_{n+i} =0$, then we can rewrite \eqref{v'[i] in reverse} as
\[ \mathbf{\mathbf{\overline{v}}}^{[i-1]}=\left(\mathbf{\overline{v}}^{[i]}_{n+i} \cdot \nabla \overline{\varphi}_i(v_1,...,v_{n+\mu})\right)^T + \mathbf{\mathbf{\overline{v}}}^{[i,*]}.\] 
The expression on the right is the analogue of the term 
\[\nabla  \varphi_i(v_{i_1},...,v_{i_{n_i}} ) \cdot \left( \begin{array}{c} v'_{i_1} \\ \vdots \\ v'_{i_{n_i}} \end{array} \right), \] 
which appears in the process of Forward AD, where the main difference is the appearance of the added vector $\mathbf{\mathbf{\overline{v}}}^{[i,*]}$.

Note that, in contrast to Forward AD, the sequence of evaluation trace pairs $[\mathbf{\mathbf{v}}^{[i]},\mathbf{\overline{v}}^{[i]}]$ appears in reverse order (that is, $[\mathbf{\mathbf{v}}^{[\mu]},\mathbf{\overline{v}}^{[\mu]}],...,[\mathbf{\mathbf{v}}^{[1]},\mathbf{\overline{v}}^{[1]}]$). In particular, 
unlike to Forward AD, it is not efficient to overwrite the previous pair in each computational step. Indeed, since the state vector $\mathbf{v}^{[i]}$ is needed to compute $\mathbf{\overline{v}}^{[i]}$, 
the pairs $[\mathbf{\mathbf{v}}^{[i]},\mathbf{\overline{v}}^{[i]}]$ are not computed (as pairs) at all. Instead, one first evaluates the evaluation trace $\mathbf{\mathbf{v}}^{[1]},...,\mathbf{\mathbf{v}}^{[\mu]}$, stores these values, and then uses them to compute the $\mathbf{\overline{v}}^{[\mu]},...,\mathbf{\overline{v}}^{[1]}$ afterwards.

\begin{example}\label{example reverse}
Consider the function 
\[f: \R \to \R^2, \ \ \ \textrm{with} \ \ f(x)= \left( \begin{array}{c} x \\ \exp(x)*\sin(x) \end{array} \right).\]
We want to determine $(y'_1 \ \ y'_2) \cdot J_f(c)$ for fixed $\mathbf{\overset{\leftharpoonup}{y}} = (y'_1 \ \ y'_2) \in \R^{1 \times 2}$ and $\mathbf{c}= c \in \R$.

Set $H = \R^5$ and $f = P_Y \circ \Phi_4 \circ \Phi_3 \circ \Phi_2 \circ \Phi_1 \circ P_X $ with 
\begin{align*}
&P_X: \R \to \R^5, \ \ \textrm{with} \ \ P_X(x) = (x, 0, 0, 0, 0) ,\\[1ex]
&\Phi_1: \R^5 \to \R^5, \ \ \textrm{with} \ \  \Phi_1\left(v_1,v_2,v_3,v_4,v_5\right) 
= (v_1,\exp(v_1),v_3,v_4,), \\ 
&\Phi_2: \R^5 \to \R^5, \ \ \textrm{with} \ \  \Phi_2\left(v_1,v_2,v_3,v_4,v_5\right)
= (v_1, v_2, \sin(v_1),v_4,v_5 ),\\[1ex]
&\Phi_3: \R^5 \to \R^5, \ \ \textrm{with} \ \  \Phi_3\left(v_1,v_2,v_3,v_4,v_5\right) = (v_1, v_2, v_3, v_1, v_5),\\[1ex]
&\Phi_4: \R^5 \to \R^5, \ \ \textrm{with} \ \  \Phi_4\left(v_1,v_2,v_3,v_4,v_5\right)
= (v_1, v_2, v_3, v_4, v_2 * v_3),\\[1ex]
&P_Y: \R^5 \to \R, \ \ \textrm{with} \ \ P_Y(v_1,v_2,v_3,v_4,v_5) = (v_4,v_5).
\end{align*}
Clearly, we obtain the evaluation trace $\mathbf{v}^{[0]}(c),...,\mathbf{v}^{[4]}(c)$ with
\[ \mathbf{v}^{[0]}(c)=\left( \begin{array}{c} c \\ 0 \\ 0 \\ 0 \\ 0 \end{array} \right), ..., 
 \mathbf{v}^{[4]}(c)=\left( \begin{array}{c} c \\ \exp(c) \\ \sin(c) \\ c \\ \exp(c)*\sin(c) \end{array} \right).\]
The Reverse Mode of Automatic Differentiation produces now the vectors\\ $\mathbf{\overline{v}}^{[4]},...,\mathbf{\overline{v}}^{[0]}$ 
with\\[1ex]
$\mathbf{\overline{v}}^{[4]}=\left( \begin{array}{c} 0 \\ 0 \\ 0 \\ y'_1 \\ y'_2 \end{array} \right),
\mathbf{\overline{v}}^{[3]}=\left( \begin{array}{c} 0 \\ y'_2\cdot \sin(c) \\ y'_2\cdot \exp(c) \\ y'_1 \\ 0 \end{array} \right),
\mathbf{\overline{v}}^{[2]}=\left( \begin{array}{c} y'_1 \\ y'_2\sin(c) \\ y'_2\exp(c) \\ 0 \\ 0 \end{array} \right)$,\\[2ex]
$\mathbf{\overline{v}}^{[1]}=\left( \begin{array}{c} y'_2\exp(c) \cdot \cos(c) + y'_1\\ y'_2\sin(c) \\ 0 \\ 0 \\ 0 \end{array} \right)$ and finally\\[2ex]
$\mathbf{\overline{v}}^{[0]}=\left( \begin{array}{c} y'_2\sin(c) \cdot \exp(c) +\big( y'_2\exp(c)\cos(c)+y'_1\big) \\ 0 \\ 0 \\ 0 \\ 0 \end{array} \right)$\\[1ex]
Then 
\[(y'_1 \ \ y'_2) \cdot J_f(c)= P_X^T \cdot \mathbf{\overline{v}}^{0} = y'_1 + y'_2\exp(c)(\sin(c) + \cos(c)).\]
\end{example}
\begin{figure}
\centering
\includegraphics[]{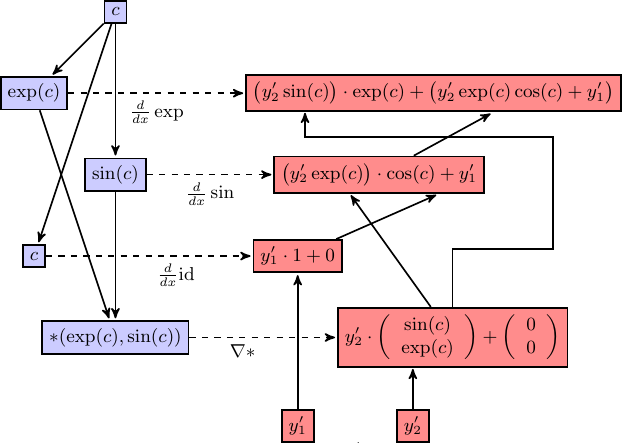}
\caption{Computational graph for Example \ref{example reverse} with elements of $ \mathbf{v}^{[4]}$ in blue and the evaluation of the directional derivative in red.}\label{graph for reverse}
\end{figure}

\begin{figure}
\includegraphics[scale=0.8]{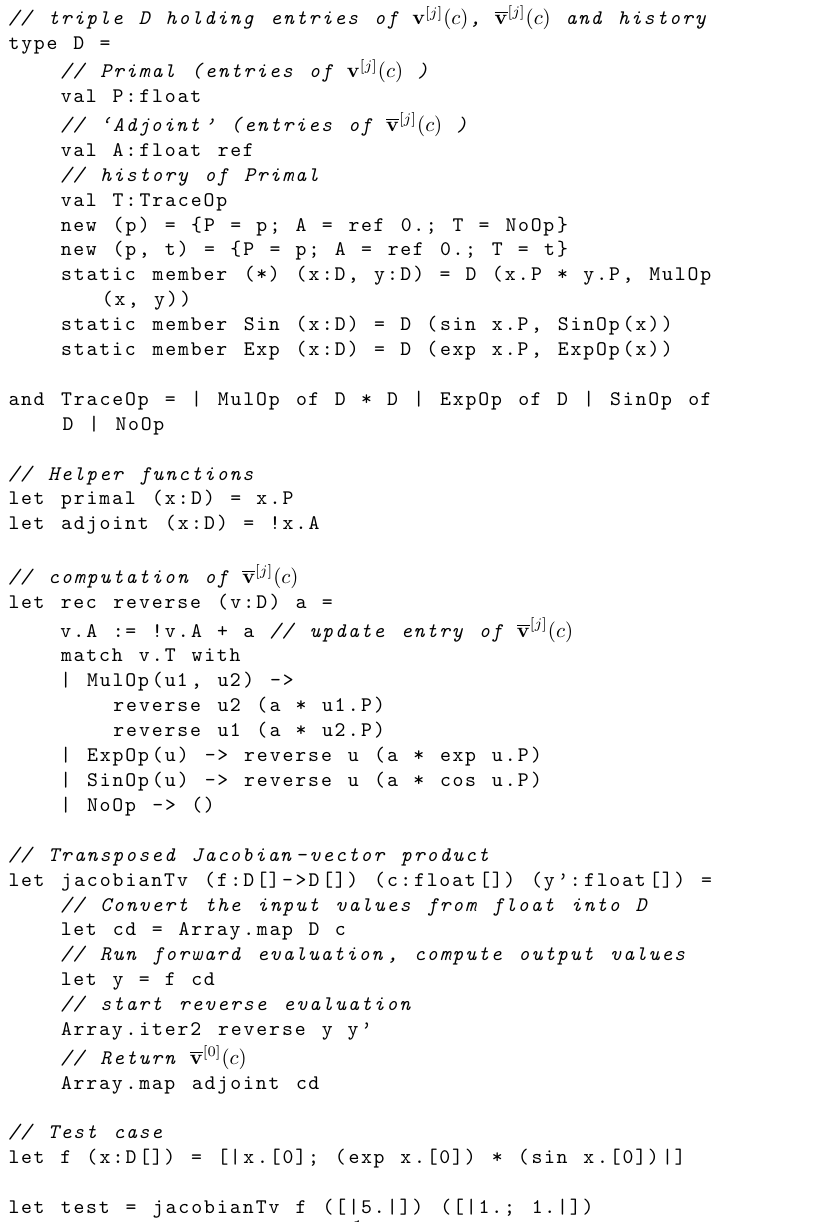}
\caption{Minimal (not optimal!) F\# example, similar to work in the library DiffSharp \cite{Gunes}, \cite{Gunes2}, showing the implementation of the Reverse Mode for Example \ref{example reverse} with the test case $c=5$, $\mathbf{\overset{\leftharpoonup}{y}} = (1 \ 1)$. (An optimal version would account for fan-out at each node.)}\label{Figure Fsharp Reverse}
\end{figure}

We summarize:
\begin{theorem}
By the above, given $\mathbf{c} \in X \subset \R^n$ and $\mathbf{\overset{\leftharpoonup}{y}} \in \R^{1 \times m}$, the evaluation of 
$\mathbf{\overset{\leftharpoonup}{y}}  \cdot J_f(\mathbf{c})$ for an automatically differentiable function $f: X \to \R^m$
can be achieved by computing the vectors $\mathbf{\mathbf{v}}^{[0]},...,\mathbf{\mathbf{v}}^{[\mu]}$ and $\mathbf{\overline{v}}^{[\mu]},...,\mathbf{\overline{v}}^{[0]}$, where the computation of each $\mathbf{\overline{v}}^{[i-1]}$ is effectively the computation of the real numbers 
\begin{align*}
&\mathbf{\overline{v}}^{[i]}_{n+i} \cdot \frac{\partial \varphi_i}{\partial  v_k}(v_{i_1}(\mathbf{c}),...,v_{i_{n_i}}(\mathbf{c}) ) + 
 \mathbf{\overline{v}}^{[i]}_{k}, \ \ \ k \neq n+i,\\[1ex]
 \textrm{and} \ \ \ \ \ \ \ \ \ \ \ &\mathbf{\overline{v}}^{[i]}_{n+i} \cdot \frac{\partial \varphi_i}{\partial  v_{n+i}}(v_{i_1}(\mathbf{c}),...,v_{i_{n_i}}(\mathbf{c}) ).
\end{align*}
\end{theorem}

Comparing the complexity of Reverse AD with the one of symbolic differentiation for the examples given in Section \ref{section comparison FAD with symbolic} gives similar results as the comparison of Forward AD with symbolic differentiation. (See, for example, Figure \ref{second graph} for the case of a composition of $3$ uni-variate functions.)

\noindent
\textbf{Acknowledgements:}
I like to thank Felix Filozov for his advice and his help with the code in Figure \ref{Figure Haskell code Dual}, At{\i}l{\i}m 
G\"{u}ne\c{s} Baydin for his advice and his help with the code in Figure \ref{Figure Fsharp Reverse} and Barak Pearlmutter for general valuable advice. I am further very thankful to the anonymous referees for valuable criticism and suggestions which led to a significant improvement of this paper. 
\
\\

\noindent
\textbf{Funding:}
This work was supported by Science Foundation Ireland grant\\ 09/IN.1/I2637. 

\newpage

%\bibliography{PrincADbib}

\begin{thebibliography}{99}
\bibitem{Berz} M. Berz, Differential algebraic description of beam dynamics to very high orders, \emph{Particle Accelerators} 24 (1989), 
109--124.
\bibitem{Bischof} C. H. Bischof. On the Automatic Differentiation of Computer Programs and an Application to Multibody Systems, 
\emph{IUTAM Symposium on Optimization of Mechanical Systems Solid Mechanics and its Applications} 43 (1996),\\ 41--48. 
\bibitem{Clifford1873} W. K. Clifford, Preliminary Sketch of Bi-quaternions, \emph{Proceedings of the London Mathematical Society} 4 (1873), 381--395.
\bibitem{Dixon} L. Dixon, Automatic Differentiation: Calculation of the Hessian, In: Encyclopedia of Optimization, second edition, Springer Science+Business Media, LLC., 2009, 133--137.
\bibitem{Garczynski} V. Garczynski, Remarks on differential algebraic approach to particle beam optics by M. Berz, \emph{Nuclear Instruments and Methods in Physics Research Section A}, 334 (2-3) (1993), 294--298.
\bibitem{GowerMello} R. M. Gower, M. P. Mello, A new framework for the computation of Hessians, \emph{Optimization Methods and Software}
27 (2) (2012), 251--273. 
\bibitem{Griewank1992ALG} A. Griewank, Achieving Logarithmic Growth of Temporal and Spatial Complexity in Reverse Automatic Differentiation, 
                        \emph{Optimization Methods and Software} 1 (1992), 35--54.
\bibitem{Griewank2003AMV} A. Griewank, A Mathematical View of Automatic Differentiation, \emph{Acta Numerica} 12 (2003), 321--398.
\bibitem{Griewank2008} A. Griewank and A. Walther, Evaluating Derivatives: Principles and Techniques of Algorithmic Differentiation, second edition, SIAM, Philadelphia, PA, 2008.
\bibitem{Gunes} A. G. Baydin, B. A. Pearlmutter, DiffSharp: Automatic Differentiation Library, version of 17th June 2015, \url{http://diffsharp.github.io/DiffSharp/}.
\bibitem{Gunes2} A. G. Baydin, B. A. Pearlmutter, A. A. Radul, J. M. Siskind, Automatic differentiation and machine learning: a survey. arXiv preprint. arXiv:1502.05767 (2015).
\bibitem{HuLu} J. H. Hubbard and B. E. Lundell, A First Look at Differential Algebra, \emph{American Mathematical Monthly} 118 (3) (2011), 245--261.
\bibitem{John1975} F. John, Partial Differential Equations, second edition, Springer-Verlag, New York-Heidelberg-Berlin, 1975.
\bibitem{kalman-2002a} D. Kalman, Double Recursive Multivariate Automatic Differentiation, \emph{Mathematics Magazine} 75 (3) (2002), 187--202.
\bibitem{KarczmarczukIII} J. Karczmarczuk, Functional Coding of Differential Forms, \emph{Proc First Scottish Workshop on Functional Programming, Stirling, Scotland, 1999}.  
\bibitem{KarczmarczukII} J. Karczmarczuk, Functional Differentiation of Computer Programs, \emph{Proc of the III ACM SIGPLAN International Conference on Functional Programming, Baltimore, MD, 1998}, 195--203.
\bibitem{KarczmarczukI} J. Karczmarczuk, Functional Differentiation of Computer Programs, \emph{Higher-Order and Symbolic Computation} 
14 (1) (2001), 35--57. 
\bibitem{Lang} S. Lang, Introduction to Differentiable Manifolds, second edition, Springer-Verlag, New York-Heidelberg-Berlin, 2002.
\bibitem{Manzyuk-mfps2012} O. Manzyuk, A Simply Typed $\lambda$-Calculus of Forward Automatic Differentiation, \emph{Electronic Notes in Theoretical Computer Science} 286 (2012), 257--272.
\bibitem{pearlmutter-siskind-popl-2007} B. A. Pearlmutter and J. M. Siskind, Lazy Multivariate Higher-Order Forward-Mode {AD}, 
\emph{Proc of the 2007 Symposium on Principles of Programming Languages, Nice, France, 2007}, 155--160.
\bibitem{PearlmutterSiskind2008a} B. A. Pearlmutter and J. M. Siskind, Reverse-Mode {AD} in a Functional Framework: Lambda the Ultimate Backpropagator, \emph{TOPLAS} 30 (2) (2008), 1--36.
\bibitem{Rall1983Dag} L. B. Rall, Differentiation and generation of Taylor coefficients in Pascal-SC, In: A New Approach to Scientific Computation, Academic Press, New York, 1983, 291--309.
\bibitem{Rall1986} L. B. Rall, The Arithmetic of Differentiation, \emph{Mathematics Magazine} 59, (1986), 275--282.
\bibitem{Ritt} J. Ritt, Differential Algebra, American Mathematical Society Colloquium Publications, vol. XXXIII, American Mathematical 
Society, New York, 1950.
\bibitem{Siskind2008NFM} J. M. Siskind and B. A. Pearlmutter, Nesting Forward-Mode {AD} in a Functional Framework, \emph{Higher-Order and Symbolic Computation} 21 (4) (2008),361--376.
\bibitem{Wengert1964ASA} R. Wengert, A Simple Automatic Derivative Evaluation Program, \emph{Communications of the ACM} 7 (8) (1964), 463--464.
\end{thebibliography}

\bigskip

\hrule
\begin{quote}
Philipp Hoffmann\\ 
National University of Ireland, Maynooth\\
Department of Computer Science\\
Maynooth, County Kildare\\
Ireland\\
\begin{tabular}{@{}l@{ }l}
\emph{email:} & \url{philipp.hoffmann@cs.nuim.ie} \\
 & \url{philip.hoffmann@maths.ucd.ie}
\end{tabular}
\end{quote}

\end{document}